\documentclass[12pt]{amsart}

\usepackage{amsfonts}
\usepackage{amssymb}
\usepackage{amsthm}
\usepackage{amsmath}
\usepackage{enumitem}
\usepackage{graphicx}
\newtheorem{theorem}{Theorem}[section]
\newtheorem{lemma}[theorem]{Lemma}
\newtheorem{remark}[theorem]{Remark}
\newtheorem{claim}[theorem]{Claim}
\newtheorem{proposition}[theorem]{Proposition}
\newtheorem{conj}[theorem]{Conjecture}

\newtheorem{corollary}[theorem]{Corollary}
\newtheorem{definition}[theorem]{Definition}
\numberwithin{equation}{section}

\newcommand{\A}{{\mathcal{A}}}

\newcommand{\eps}{\varepsilon}

\newcommand{\Balpha}{\mbox{$\hspace{0.1em}\rule[0.01em]{0.05em}{0.39em}\hspace{-0.21em}\alpha$}}

\begin{document}
\title[Minimal two-spheres in three-spheres]{Minimal two-spheres in three-spheres}
\author{Robert Haslhofer}\address{Department of Mathematics\\ University of Toronto\\Toronto, ON M5S 2E4}
\email{roberth@math.toronto.edu}

\author{Daniel Ketover}\address{Department of Mathematics\\Princeton
University\\Princeton, NJ 08544}
 \email{dketover@math.princeton.edu}
\maketitle\thanks

\begin{abstract}
We prove that any manifold diffeomorphic to $\mathbb{S}^3$ and endowed with a generic metric contains at least two embedded minimal two-spheres.  The existence of at least one minimal two-sphere was obtained by Simon-Smith in 1983.  Our approach combines ideas from min-max theory and mean curvature flow.  We also establish the existence of smooth mean convex foliations in three manifolds.   We apply our methods to solve a problem posed by S.T. Yau in 1987 on whether the planar two-spheres are the only minimal spheres in ellipsoids centered about the origin in $\mathbb{R}^4$.   Finally, considering the example of degenerating ellipsoids we show that the assumptions in the multiplicity one conjecture and the equidistribution of widths conjecture are in a certain sense sharp.
\end{abstract}

\section{Introduction}
The min-max method goes back to Birkhoff who in 1917 proved the following theorem:

\begin{theorem}[{G.D. Birkhoff 1917 \cite{Birkhoff}}]\label{birkhoff}
Any closed Riemannian two-sphere contains at least one closed geodesic. 
\end{theorem}

The difficulty in proving Theorem \ref{birkhoff} is that a two-sphere is simply connected, and thus minimization methods to produce minimizing critical points for the length functional do not work.

Birkhoff's idea was to use Morse theoretic arguments to produce higher index (i.e. unstable) geodesics.  Loosely speaking, he considered sweep-outs of the two-sphere by closed curves, and argued that the longest slice in a sweep-out that is ``pulled tight" is a closed geodesic.   Birkhoff left open the question of whether the geodesic produced by Theorem \ref{birkhoff} is embedded (i.e. does not have self-intersections).   Using heat flow methods, Grayson \cite{Grayson} later proved the existence of an embedded closed geodesic in the setting of Theorem \ref{birkhoff}.

Birkhoff considered one parameter families of curves to produce his closed geodesic.  There are also higher non-trivial families of curves one can consider to produce more geodesics.  Lusternik and Schnirelmann later used these families to prove:

\begin{theorem}[{L. Lusternik and L. Schnirelmann 1947 \cite{LS}}]
Any closed Riemannian two-sphere contains at least $3$ simple closed geodesics. 
\end{theorem}

In the work of Lusternik and Schnirelmann, it was not clear if the geodesics obtained are indeed simple (embedded).  This was later shown to be true by Grayson \cite{Grayson} using his heat flow methods.  

In one higher dimension, one can consider sweep-outs of three-spheres by two-spheres, and hope to produce an embedded minimal two-sphere.  In 1983, Simon and Smith carried this out (adapting the more general min-max theory of Almgren and Pitts to the case of surfaces with fixed topology) and proved:

\begin{theorem}[L. Simon and F. Smith 1983 \cite{SS}]\label{simonsmith}
Let $M$ be a Riemannian three-manifold diffeomorphic to $\mathbb{S}^3$.  Then $M$ contains an embedded minimal two-sphere.
\end{theorem}

In analogy with the case of simple closed geodesics on two-spheres, there are higher parameter families of two-spheres on three-spheres that one can consider and hope to produce additional minimal two-spheres.  Let us briefly describe these families.  Let $\mathcal{G}$ denote the space of (possibly degenerate) geodesic spheres in round $\mathbb{S}^3$ and let $\mathcal{P}\subset\mathcal{G}$ denote the subset consisting of great spheres.  Note that $\mathcal{P}$ is homeomorphic to $\mathbb{RP}^3$.  There's a natural map
\begin{equation}
\pi: \mathcal{G}\rightarrow\mathcal{P},
\end{equation}
sending a point in $\mathcal{G}$ to the great sphere parallel to it.   Thus $\mathcal{G}$ is an interval bundle over $\mathcal{P}$.

Recall that $H_i(\mathcal{P},\mathbb{Z}_2)=\mathbb{Z}_2$ for $i=0,1,2,3$, and let $a_i$ be a generator of $H_i(\mathcal{P},\mathbb{Z}_2)=\mathbb{Z}_2$.  Then the families $\pi^{-1}(a_0),\pi^{-1}(a_1),\pi^{-1}(a_2),\pi^{-1}(a_3)$ give four non-trivial sweep-outs of $\mathbb{S}^3$.  One might hope that each of the four families produces via min-max a distinct minimal two-sphere.\footnote{In fact, White exhibited certain three-spheres (Theorem 4.5 in \cite{White_posricci}) containing precisely $4$ embedded minimal two-spheres.} The major difficulty is the phenomenon of \emph{multiplicity} in min-max theory.  Namely, it could happen that the min-max spheres associated with the second, third and fourth family,  just give the sphere associated to the first family counted with higher integer multiplicities. 

To explain the min-max approach in somewhat more detail, recall that by Hatcher's solution of the Smale conjecture the space of embedded two-spheres in $S^3$ together with trivial two-spheres consisting of a point deformation retracts onto $\mathcal{G}/\partial\mathcal{G}\simeq \mathbb{RP}^4$.  We can consider the area functional on this space, and a minimal sphere is precisely a critical point of this functional. The central theorem of Lusternik-Schnirelmann theory (see e.g. Theorem 1.15 in \cite{C}) says that the number of critical points of a smooth real-valued function defined on a manifold $M$ is bounded from below by one plus the maximal cup-length (with respect to some ring) of the cohomology ring of $M$.  Since the cohomology ring of $\mathbb{RP}^4$ with $\mathbb{Z}_2$ coefficients has cup-length $4$, it follows that one expects $5$ critical points of the area functional (in other words, minimal surfaces).  Since the absolute minimizer of the area functional has zero area and is achieved by a point, Lusternik-Schnirelmann theory thus predicts the existence of $4$ non-trivial minimal two-spheres.  Again, the issue in making this sketch rigorous is that from the point of view of min-max theory, minimal spheres with differing integer multiplicities count as distinct critical points of the area functional and thus may not be geometrically distinct.

The problem of bounding the number of critical points in terms of a topological count is in some ways a Riemannian analog to the Arnol'd conjecture in symplectic geometry.  \\

Let us summarize some progress on the problem of finding minimal two-spheres (besides the one from Simon-Smith) in three-spheres:

Using degree methods, i.e. an approach different from min-max theory, White \cite{White_posricci} proved the existence of at least $2$ minimal two-spheres in the situation that $M$ has positive Ricci curvature. 

More recently, Marques-Neves \cite{MN} applied Lusternik-Schnirelmann theory to obtain the existence of at least $4$ embedded minimal surfaces in any three-manifold.  Assuming positive Ricci curvature, they obtain infinitely many minimal surfaces as conjectured by Yau. In either case, nothing is known about the topological type of their minimal surfaces, i.e. their proof does not establish the existence of any embedded minimal two-sphere besides the one from Simon-Smith.\\

In this paper, using combined efforts from min-max theory and mean curvature flow we prove the existence of at least one additional embedded minimal two-sphere besides the one from Simon-Smith:
\begin{theorem}[Existence of minimal two-spheres] \label{main}
Let $M$ be a three-manifold diffeomorphic to $\mathbb{S}^3$ and endowed with a bumpy metric.
Then $M$ contains at least $2$ embedded minimal two-spheres. More precisely, exactly one of the following alternatives holds:
\begin{enumerate}
\item $M$ contains at least $1$ stable embedded minimal two-sphere, and at least $2$ embedded minimal two-spheres of index one.
\item $M$ contains no stable embedded minimal two-sphere, at least $1$ embedded minimal two-sphere $\Sigma_1$ of index one, and at least $1$ embedded minimal two-sphere $\Sigma_2$ of index one or two.\footnote{As will be clear from the course of the proof, the second minimal two-sphere $\Sigma_2$ arises with \emph{multiplicity one} from a two-parameter min-max procedure.  Marques-Neves have announced in the Almgren-Pitts setting that when the min-max surface is obtained with multiplicity one, the index is equal to the number of parameters.  It is likely their argument would thus apply to show that the index of $\Sigma_2$ is two. }  In this case
\begin{equation}
|\Sigma_1|<|\Sigma_2|\ < 2|\Sigma_1|.\footnote{For $A$ a set, throughout this paper $|A|$ or $\mathcal{H}^2(A)$ will denote the $2$-dimensional Hausdorff measure of $A$}
\end{equation}
\end{enumerate}
\end{theorem}

In other words, Theorem \ref{main} proves the existence of at least 3 embedded minimal two-spheres if there is a stable one, and of at least 2 embedded minimal two-spheres if there is no stable one, both without any assumptions on the curvature. We recall that a metric is called \emph{bumpy} provided that its immersed minimal surfaces don't have nontrivial Jacobi fields. This is a generic property for metrics \cite{White_bumpy}. In particular, any metric can be perturbed slightly to be bumpy.

Regarding the intersection properties of the minimal two-spheres produced by Theorem \ref{main}, it is easy to see that the 3 embedded minimal two-spheres from case (1) are pairwise disjoint, but that the 2 embedded minimal two-spheres from case (2) always intersect each other.

A natural family of examples of three-spheres to test Theorem \ref{main} on are ellipsoids.  Namely, given $a> b> c> d>0$, consider the ellipsoid
\begin{equation}
E(a,b,c,d):= \left\{\frac{x_1^2}{a^2}+ \frac{x_2^2}{b^2}+\frac{x_3^2}{c^2}+\frac{x_4^2}{d^2}=1\right\}\subset\mathbb{R}^4.
\end{equation}
It contains at least 4 minimal `planar' two-spheres, which are obtained by the intersection of $E(a,b,c,d)$ with the coordinates hyperplanes $\{x_i=0\}$.  

In one of his early problem sections on minimal surfaces, S.T. Yau asked (c.f. Section 4 in \cite{Yau}):\\
\\
{\bf Yau's problem (1987):} Are the only minimal two-spheres in an ellipsoid centered about the origin in $\mathbb{R}^4$ the planar ones?
\\
\\
It is a classical theorem of Almgren \cite{Almgren} that in round spheres, the only immersed minimal two-spheres are planar.  White \cite{White_bumpy} proved that when $a$, $b$, $c$ and $d$ are close enough, then all minimal two-spheres in $E(a,b,c,d)$ are planar.  

We obtain the following negative answer to Yau's question:
\begin{theorem}[Answer to Yau's question]\label{yau}
For fixed $b,c$ and $d$, if $a$ is chosen sufficiently large, then the ellipsoid $E(a,b,c,d)$ contains a non-planar embedded minimal two-sphere.
\end{theorem}

Fix $b,c,d$ and write $E(a):=E(a,b,c,d)$.   To prove Theorem \ref{yau}, we will show that for $a$ large enough (i.e. as the ellipsoids become sufficiently elongated) the second minimal two-sphere $\Sigma_2(a)\subset E(a)$ produced by Theorem \ref{main} is not planar. In particular, for $a\gg b$ the ellipsoid $E(a)$ contains at least $5$ embedded minimal two-spheres.\\ 

The degenerating ellipsoids also are interesting with respect to other conjectures in min-max theory.  Marques-Neves conjectured that generically, the min-max process does not produce unstable components with multiplicity \cite{MN_mult}. The elongated ellipsoids are a ``critical" family of examples where the conjecture comes closer and closer to failing as the metrics degenerate. Namely, we prove the following proposition:

\begin{proposition}\label{limit}
As $a\rightarrow\infty$, the minimal two-spheres $\Sigma_2(a)$ converge as varifolds to a minimal two-sphere with multiplicity two. 
\end{proposition}

For a more detailed discussion of this, as well as for the relation with the equidistribution of width conjecture, please see Section \ref{sec_ell}.
\\
\\
Related to the study of minimal two-spheres in the three-sphere is the study of the space of embedded projective planes in $\mathbb{RP}^3$.  By the version of Smale's conjecture\footnote{Bamler-Kleiner have recently announced a proof of this conjecture.} for $\mathbb{RP}^3$, the space of embedded $\mathbb{RP}^2$'s in $\mathbb{RP}^3$ is expected to retract onto $\mathbb{RP}^3$.  As the cup-length of $\mathbb{RP}^3$ with $\mathbb{Z}_2$ coefficients is three, by Lusternik-Schnirelmann theory one expects $4$ minimal projective planes (unlike in the case of two-spheres in the three-sphere, the area minimizing projective plane \emph{does} count as a non-trivial critical point).  We obtain the following theorem:\footnote{We thank Andr\'e Neves for bringing this problem to our attention.}

\begin{theorem}[Projective Planes in $\mathbb{RP}^3$]\label{thm_intro_proj}
Let $M$ be a $3$-manifold diffeomorphic to $\mathbb{RP}^3$ endowed with a metric of positive Ricci curvature.  Then $M$ admits at least $2$ embedded minimal projective planes.  
\end{theorem}

Let us now sketch the main ideas of our proof of Theorem \ref{main}.

If $M$ admits a stable embedded minimal two-sphere, then the manifold is a kind of dumbbell. We can then consider 1-parameter sweep-outs of both halves. More precisely, using the the lower index estimates obtained by the second-named author and Y. Liokumovich \cite{KL} it follows that each of the two balls bounded by the stable two-sphere contains an unstable two-sphere of index one in its interior.

Let us now consider the case that $M$ does not contain any stable embedded minimal two-spheres. Using Simon-Smith's existence theorem we obtain $1$ embedded minimal two-sphere of index one. To obtain a second embedded minimal two-sphere we have to work much harder. To get started, sliding the Simon-Smith sphere a bit to both sides (using the lowest negative eigenfunction of the stability operator) we can decompose $M=D_1\cup Z \cup D_2$ where $Z$ is the short cylindrical region obtained by sliding the Simon-Smith sphere around, and $D_1$ and $D_2$ are smooth embedded $3$-discs with mean convex boundary. To proceed, we prove the following general theorem establishing the existence of smooth mean convex foliations in three-manifolds:

\begin{theorem}[Existence of smooth mean convex foliations in three-manifolds]\label{main_foliation_theorem}
Let $D\subset M^3$ be a smooth three-disc with mean convex boundary. Then exactly one of the following alternatives holds true:
\begin{enumerate}
\item There exists an embedded stable minimal two-sphere $\Sigma\subset \textrm{Int}(D)$.
\item There exists a smooth foliation $\{ \Sigma_t \}_{t\in [0,1]}$ of $D$ by mean convex embedded two-spheres. 
\end{enumerate}
\end{theorem}

Note that in our application we are always in case (2). We have stated Theorem \ref{main_foliation_theorem} in its more general form since it is clearly of independent interest, and will likely have several other applications in the geometry of three-manifolds. In fact, we've been informed recently that it even has applications in the field of inverse problems \cite{ABN}.

To obtain some intuition for Theorem \ref{main_foliation_theorem}, imagine (without worrying too much about the details for the moment) that the disc $D$ evolves by mean curvature flow \cite{Hui84}. Recall that mean-convexity is preserved under mean curvature flow. In the simplest possible scenario, the mean curvature flow of $D$ remains smooth and either becomes extinct in finite time in a round point, giving the foliation from (2), or converges for $t\to\infty$ to a minimal embedded two-sphere, giving (1).

Of course, in general the situation is much more complicated since the mean curvature flow typically develops local singularities. One way to continue the flow through singularities is given by the level set method \cite{ES,CGG}, and in fact our proof shows that case (2) happens if and only if the level set flow becomes extinct in finite time.
The main issue however is that the foliation produced by the level set flow is in general singular.

To produce a smooth foliation instead of a singular foliation we use mean curvature flow with surgery. Mean curvature flow with surgery in general ambient manifolds has been constructed by Brendle-Huisken \cite{BrendleHuisken}. However, since we also need a canonical neighborhood theorem for our application we instead extend the approach from Haslhofer-Kleiner \cite{HK} to the setting of general ambient manifolds. This extension, and in particular the canonical neighborhood theorem for mean curvature flow with surgery in general ambient manifolds, which is again of independent interest, are discussed in Section \ref{sec_surgery_ambient}. We then combine the existence theorem, the canonical neighborhood theorem, and methods from the recent topological application of mean curvature flow with surgery by Buzano-Haslhofer-Hershkovits \cite{BHH} to produce the desired smooth foliation.\\

Having discussed Theorem \ref{main_foliation_theorem} in quite some detail, let us now sketch how it can be used to finish the proof of Theorem \ref{main}. Recalling that $M=D_1\cup Z \cup D_2$ and using the foliations of $D_1$ and $D_2$ produced by Theorem \ref{main_foliation_theorem} we can build an optimal foliation of $M$, by which we mean a foliation $\{\Sigma_t\}_{t\in [-1,1]}$ of $M$ by two-spheres so that the Simon-Smith sphere sits in the middle of the foliation as $\Sigma_{0}$ and all other slices have less area. From the one parameter family $\{\Sigma_t\}$ we can then form a two parameter family $\{\Sigma_{s,t}\}$. Loosely speaking, $\Sigma_{s,t}$ consist of the two surfaces $\Sigma_s$ and $\Sigma_t$ joined by a very thin tube.

As long as the tube in the surface $\Sigma_{s,t}$ is very thin, it contributes negligible area and thus we obtain
\begin{equation}\label{badareas}
\sup_{s,t} |\Sigma_{s,t}| \approx 2|\Sigma_{0}|.
\end{equation}
Using the catenoid estimate from Ketover-Marques-Neves \cite{KMN}, by opening up the neck near $(s,t)\approx(0,0)$ (where the area of $\Sigma_{s,t}$ is close to $2|\Sigma_{0}|$) we can improve \eqref{badareas} to show
\begin{equation}\label{upperbound}
\sup_{s,t} |\Sigma_{s,t}| < 2|\Sigma_{0}|.
\end{equation}
The bound \eqref{upperbound} guarantees that the minimal surface produced from two-parameter min-max is not simply twice the minimal sphere $\Sigma_{0}$ realizing the width of $M$. Finally, by Lusternik-Schnirelmann theory, the minimal surface produced cannot be $\Sigma_{0}$ with multiplicity one. This finishes the sketch of our existence proof for the second embedded minimal two-sphere.

Finally, we make a few remarks about the necessity of using the mean curvature flow in this paper.  Assuming the manifold has no stable spheres, we can use min-max arguments to obtain an optimal \emph{sweepout} with the Simon-Smith sphere $\Sigma_0$ sitting with largest area in the middle.  To form the two-parameter family $\Sigma_{s,t}$ however we need to take the connect sum of $\Sigma_s$ and $\Sigma_t$, and for this it is crucial that $\Sigma_s$ and $\Sigma_t$ are disjoint for $s\neq t$.  For this reason we cannot use an optimal \emph{sweepout} but rather need an optimal \emph{foliation}, which the mean curvature flow provides.  \\

This article is organized as follows. In Section \ref{sec_topo}, we discuss the topology of the space of embedded two-spheres. In Section \ref{sec_optimal}, we show how to produce optimal foliations using Theorem \ref{main_foliation_theorem} as a black-box. In Section \ref{sec_twopar}, we construct the two-parameter sweep-out $\{\Sigma_{s,t}\}$. In Section \ref{sec_main}, we show how all the ingredients can be combined to prove Theorem \ref{main} (using Theorem \ref{main_foliation_theorem}). In Section \ref{sec_fol_eucl}, we prove the existence of smooth mean convex foliations in Euclidean space. In Section \ref{sec_surgery_ambient}, we extend the surgery construction from Haslhofer-Kleiner to general ambient manifolds. In Section \ref{sec_fol_amb}, we extend the construction of smooth mean convex foliations to general ambient manifolds and thus conclude the proof of Theorem \ref{main_foliation_theorem}. In Section \ref{sec_ell}, we analyze the examples of ellipsoids to answer Yau's question and discuss the relationship with the multiplicity one conjecture and the equidistribution of width conjecture of Marques-Neves.
Finally, in Section \ref{section_projective} we prove Theorem \ref{thm_intro_proj}.
\\

{\bf Acknowledgements:} R.H. was partially supported by NSERC grant RGPIN-2016-04331, NSF grant DMS-1406394 and a Connaught New Researcher Award. D.K. was partially supported by an NSF Postdoctoral Research fellowship as well as ERC-2011-StG-278940. D.K. would like to thank A. Neves and F. Cod\'a-Marques for their encouragement and support, and for conversations from which some ideas in this paper first arose.   We would also like to thank Y. Liokumovich and N. Sarquis for several useful conversations.

\section{Topology of Space of embedded two-spheres}\label{sec_topo}

In order to apply min-max theory to the space of embedded two-spheres in a three-sphere, we need to first understand the topological type of this space.  

 Letting $\mathbb{S}^3$ denote the round three-sphere in $\mathbb{R}^4$, let us consider the geodesic spheres
\begin{equation}
\mathcal{G}(a,b,c,d,e):=\{ax_0+bx_1+cx_2+dx_3=e\}\cap\mathbb{S}^3,
\end{equation}
and the space of all (also degenerate) geodesic spheres:
\begin{equation}
\mathcal{G}:=\bigcup_{a,b,c,d,e\in\mathbb{R}} \mathcal{G}(a,b,c,d,e).
\end{equation}
Note that $\partial\mathcal{G}$ consists of the degenerate (i.e. point) spheres and thus is identified naturally with $\mathbb{S}^3$.

We will now define a space $\mathcal{S}$ of embedded two-spheres (including some controlled degenerations) in $S^3$. To this end, consider the spaces
\begin{equation}
\mathcal{X}:= \{\phi(\mathbb{S}^2)\;|\; \phi:\mathbb{S}^2\rightarrow\mathbb{S}^3 \mbox{ is a smooth embedding}\}
\end{equation}
and
$$\mathcal{Y}:= \{\phi(\mathbb{S}^2)\;|\; \phi:\mathbb{S}^2\rightarrow\mathbb{S}^3 \mbox{ is a smooth map whose image is a $1$d graph}\}.$$

Let us endow $\mathcal{X}\cup\mathcal{Y}$ with the flat metric.  To turn $\mathcal{X}\cup\mathcal{Y}$ into a metric space, we thus identify points of $\mathcal{Y}$ to a single point to obtain the quotient space $\mathcal{S}:=(\mathcal{X}\cup\mathcal{Y})/\mathcal{Y}$.

It follows from Hatcher's proof of Smale's conjecture (see item (14) in the Appendix \cite{Hatcher}) that the space $\mathcal{S}$ is homotopy equivalent to $\mathcal{G}/\partial\mathcal{G}$.

The space $\mathcal{G}$ is homeomorphic to $\mathbb{RP}^4\setminus B$ where $B$ is an open ball.  One can see this as follows.  Consider the subset of $\mathcal{P}\subset\mathcal{G}$ consisting of the great spheres, i.e., the intersection of $\mathbb{S}^3$ with a $3$-plane in $\mathbb{R}^4$ passing through the origin.  The space $\mathcal{P}$ is homeomorphic to $\mathbb{RP}^3$.  Recall from the Introduction that there is a natural map
\begin{equation}
\pi:\mathcal{G}\rightarrow\mathcal{P}
\end{equation}
sending a geodesic sphere to the great sphere it is parallel to.  This exhibits $\mathcal{G}$ as a twisted interval bundle over $\mathbb{RP}^3$, which gives that $\mathcal{G}$ is homeomorphic to $\mathbb{RP}^4\setminus B$. 

Combining the above, we see that $\mathcal{G}/\partial\mathcal{G}$ is obtained from $\mathbb{RP}^4\setminus B$ by identifying $\partial B$ to a point, and thus that $\mathcal{S}$ is homotopy equivalent to $\mathbb{RP}^4$. It follows that the homology groups are given by
\begin{equation}
H_k(\mathcal{S}, \mathbb{Z}_2)=\mathbb{Z}_2\mbox{ for each } k=1,2,3,4.
\end{equation}

For the purposes of Lusternik-Schnirelmann theory, it will be useful to also consider the cohomology ring, which is given by
\begin{equation}
H^\ast( \mathcal{S}, \mathbb{Z}_2)=\mathbb{Z}_2[\alpha]/(\alpha^5),
\end{equation}
where $\alpha$ denotes a generator of $H^1(\mathcal{S},\mathbb{Z}_2)$.

The group $H_1(\mathcal{S}, \mathbb{Z}_2)$ is generated by choosing a great sphere and considering the geodesic spheres parallel to it. The group $H_2(\mathcal{S}, \mathbb{Z}_2)$ is generated by the two-parameter family of spheres consisting of a non-trivial loop of great spheres in $\mathcal{P}$ together with the parallel geodesic spheres to each great sphere in the loop.  More explicitly, these two groups are represented by the one-parameter family
\begin{equation}
\mathbb{RP}^1\rightarrow \mathcal{G}, \qquad
[a_0,a_1]\mapsto \mathcal{G}(a_0,0,0,0,a_1),
\end{equation}
and the two-parameter family
\begin{equation}
\mathbb{RP}^2\rightarrow \mathcal{G},\qquad
[a_0,a_1,a_2]\mapsto \mathcal{G}(a_0,a_1,0,0,a_2).
\end{equation}

Now we will consider general sweep-outs into a three-sphere.  Let  $X$ be a simplicial complex and let
\begin{equation}
\Phi:X\rightarrow\mathcal{S}
 \end{equation}
be a sweep-out.
 Let $\omega$ denote a cohomology class in $H^*(\mathcal{S},\mathbb{Z}_2)$. Then we say $\Phi$ \emph{detects} $\omega$ if 
 \begin{equation}
\Phi^*(\omega)\neq 0.
 \end{equation}
 
Recall that $\alpha$ denotes a generator of $H^1(\mathcal{S},\mathbb{Z}_2)$.  The sweep-out $\Phi$ detecting $\alpha^i\in H^i(\mathcal{S}, \mathbb{Z}_2)$ is equivalent to the following (c.f. Section 4.1 in \cite{MN}):  there exists a cohomology class $\lambda$ in $H^1(X,\mathbb{Z}_2)$ with $\lambda^i$ non-zero so that if $\gamma:S^1\rightarrow X$ is a closed curve, then 
 $\Phi\circ\gamma:S^1\rightarrow\mathcal{S}$ is a nontrivial sweep-out of $M$ if and only if $\lambda(\gamma)\neq 0$.
 
 In particular, a smooth family of surfaces parameterized by $\mathbb{S}^1$ detects the generator $\alpha$ if and only if it is a non-trivial sweep-out of $M$.  
  
 For each $i\in\{1,2,3,4\}$ let $S_i$ be the set of all sweep-outs (parameterized by some complex) which detect $\alpha^i$.
  For each $i\in\{1,2,3,4\}$, we then define the min-max widths:
 \begin{equation}\label{def_width}
 \omega_i(M):= \inf_{\Phi\in S_i} \sup_{x\in Dom(\Phi)} |\Phi(x)|,
 \end{equation} 
 where $\mbox{Dom}(\Phi)$ denotes the domain of $\Phi$. 
Note from the definition that if a sweep-out detects $\alpha^i$, it detects $\alpha^j$ for any $1\leq j\leq i$.  It follows from this that $\omega_1(M)\leq \omega_2(M)\leq\omega_3(M)\leq \omega_4(M)$.  By the Isoperimetric Inequality (c.f. Appendix A in \cite{CD}) it follows that $\omega_1(M)>0$. 

 \begin{theorem}[Min-Max Theorem]\label{minmax}
 Let $M$ be a Riemannian $3$-sphere.  For each $i\in\{1,2,3,4\}$ is associated a stationary integral varifold $V_i$ 
  \begin{equation}
 V_i = \sum_{j=1}^{k_i} m_i^j \Sigma_i^j,
 \end{equation}
 \noindent
where for each $i$, $\{\Sigma_i^j\}_{j=1}^{k_i}$ is a collection of pairwise disjoint embedded minimal two-spheres and $m_i^j$ are positive integer multiplicities.

Moreover, there holds
 \begin{equation}\label{multi}
\omega_i(M)=|V_i|= \sum_{j=1}^{k_i} m_i^j |\Sigma_i^j|.
 \end{equation}
 and
 \begin{equation}
 0<\omega_1(M)\leq \omega_2(M)\leq\omega_3(M)\leq \omega_4(M).
 \end{equation}
We have the following index bounds for each $i\in\{1,2,3,4\}$:
\begin{equation}\label{indexbound}
\sum_{j=1}^{k_i} \mbox{index}(\Sigma^j_i) \leq i.
\end{equation}
\end{theorem}

The existence of the minimal surface realizing $\omega_1(M)$ in Theorem \ref{minmax} was proved by Simon-Smith \cite{SS} in 1983.  Simon-Smith explicitly considered smooth sweepouts with integer coefficients but the situation considered here with $\mathbb{Z}_2$ coefficients is the same.  The necessary changes for the spheres arising from higher parameter sweepouts are straightforward (see for instance the Appendix of \cite{CGK}).   The upper bound on the Morse index \eqref{indexbound} follows from work of Marques-Neves \cite{MN}.

\begin{remark}
Because of the possibility of integer multiplicities, note that Theorem \ref{minmax} only guarantees the existence of one minimal two-sphere $\Sigma$, as it may happen that $V_1=\Sigma$, $V_2=2\Sigma$, $V_3=3\Sigma$ and $V_4=4\Sigma$. 
It is natural to conjecture that for generic metrics, the multiplicities in \eqref{multi} are all $1$, and that Theorem \ref{minmax} should produce at least $4$ distinct embedded minimal two-spheres as predicted by Lusternik-Schnirelmann theory.
\end{remark}

\section{Optimal foliations}\label{sec_optimal}
Let $M$ be a manifold diffeomorphic to the three-sphere. 
We call a one parameter family of sets $\{\Sigma_t\subset M\}_{t\in [-1,1]}$ an \emph{optimal sweep-out} of $M$ by two-spheres if
\begin{enumerate}
\item $\Sigma_t$ is a smooth embedded two-sphere for all $t\in (-1,1)$
\item $\Sigma_{-1}$ and $\Sigma_1$ are one-dimensional graphs
\item $\Sigma_{0}$ is an unstable minimal two-sphere whose area realizes the $1$-width $\omega_1(M)$
\item $|\Sigma_t| < |\Sigma_{0}|$ for all $t\neq 0$
\item $\Sigma_t=\exp_{x\in \Sigma_0}(t\phi(x)\nu(x))$ for $t$ near $0$, where $\phi$ is the normalized first eigenfunction of the stability operator $L_{\Sigma_0}$. In particular, $|\Sigma_t| \leq |\Sigma_{0}|-ct^2$ for some $c>0$ and $t$ near $0$
\item $\Sigma_t$ varies smoothly for $t\in (-1,1)$ and
\item $\Sigma_t\rightarrow\Sigma_{\pm 1}$ in the Hausdorff topology as $t\rightarrow \pm 1$.
\end{enumerate}

We call an optimal sweep-out by two-spheres an \emph{optimal foliation} by two-spheres if in addition $\Sigma_s$ and $\Sigma_t$ are disjoint whenever $s\neq t$.\\

The goal of this section is to prove the following theorem:

\begin{theorem}[Existence of optimal foliations]\label{optimal}
Let $M$ be a Riemannian three-manifold diffeomorphic to $\mathbb{S}^3$.  Suppose $M$ admits no stable minimal two-spheres.  Then $M$ admits an optimal foliation by two-spheres. 
\end{theorem}

Recall that the 1-width $\omega_1(M)$ is defined in \eqref{def_width} as min-max value over sweep-outs detecting $\alpha$. Let us also consider the quantity
\begin{equation}\label{mintwosphere}
\gamma(M):= \inf_{\Sigma\in\mathcal{S}_{\min}} |\Sigma |,
\end{equation}
where $\mathcal{S}_{\min}$ denotes the space of embedded minimal two-spheres in $M$.

We need the following lemma (c.f. Lemma 3.5 in \cite{MN2}):
\begin{lemma}[{Realizing the 1-width}]\label{lemma_width}
Let $M$ be a three-sphere containing no stable embedded two-spheres.  
Then the infimum in \eqref{mintwosphere} is achieved by an index one minimal two-sphere $\Sigma$ with multiplicity one. Moreover, $\Sigma$ realizes the $1$-width of $M$, in particular $\gamma(M)=\omega_1(M)$.
\end{lemma}

\begin{proof}
By Simon-Smith's Theorem (i.e. the $i=1$ case in Theorem \ref{minmax}) the set $\mathcal{S}_{\min}$ is non-empty and thus \eqref{mintwosphere} is well-defined. Since the monotonicity formula gives a lower bound for the area of minimal surfaces, we see that $\gamma(M)>0$.

If $\mathcal{S}_{\min}$ is a finite set, then the infimum \eqref{mintwosphere} is achieved by a minimal two-sphere $\Sigma$.  
If $\mathcal{S}_{\min}$ is infinite we claim the infimum is still attained.  To see this, let $\Sigma_i\in\mathcal{S}_{\min}$ be a minimizing sequence with $|\Sigma_i|\rightarrow\gamma(M)$.  Since the areas of $\Sigma_i$ are uniformly bounded, after passing to a subsequence they converge to a smooth minimal two-sphere $\Sigma$, potentially with multiplicity $k$.  If $k>1$, then $|\Sigma|\leq |\Sigma_i|/k$ which contradicts the fact that $\Sigma_i$ was a minimizing sequence.  Thus $k=1$ and $\Sigma_i\rightarrow \Sigma$ smoothly.   It follows that $|\Sigma|=\gamma(M)$.

Since $M$ contains no stable minimal two spheres, $\Sigma$ is unstable.  It then follows from Lemma 3.5 in \cite{MN2}, that $\Sigma$ must have index one.\\
Observe that $\omega_1(M)\geq \gamma(M)$, since otherwise by Simon-Smith's Theorem (i.e. the $i=1$ case in Theorem \ref{minmax}) we could produce a minimal-two sphere with smaller area contradicting the definition of $\gamma(M)$.

Finally, the argument below (which is proven without using the `moreover'-part of the lemma) shows that $\Sigma$ is the maximal slice of an optimal sweep-out of $M$ by two spheres, and thus $\Sigma$ realizes $\omega_1(M)$, and in particular $\omega_1(M)\leq \gamma(M)$.
\end{proof}

\begin{proof}[{Proof of Theorem \ref{optimal} (using Theorem \ref{main_foliation_theorem})}]
By assumption $M$ does not contain any stable minimal two-spheres. Thus, by Lemma \ref{lemma_width} there exists a minimal two-sphere $\Sigma$ which realizes the infimum in \eqref{mintwosphere}, and moreover has index one and multiplicity one.

Choose a unit normal $\nu$ on $\Sigma$. Let $\phi$ denote the lowest eigenfunction of the stability operator $L_\Sigma$ with eigenvalue $\lambda$, and normalized in $L^2$. Since the lowest eigenfunction doesn't change sign, we can assume that $\phi$ is positive.
Consider the family
\begin{equation}
\Sigma_t:= \exp_{x\in\Sigma}(t\phi(x)\nu(x)),
\end{equation}
where the parameter $t$ is small enough. Note that for $\eps$ small enough $\{\Sigma_t\}_{t\in [-\eps,\eps]}$ gives a foliation by smooth embedded two-spheres of a cylindrical neighborhood $Z$ of $\Sigma$. Moreover, Taylor expansion gives
\begin{equation}\label{areas2}
|\Sigma_t|=|\Sigma|-\frac{\lambda}{2} t^2+\mathcal{O}(t^3),
\end{equation}  
and
\begin{equation}
H_{\Sigma_t}= \lambda \phi t+\mathcal{O}(t^2).
\end{equation}
Thus, if we choose $\eps$ small enough we obtain a decomposition
\begin{equation}
N=D_1 \cup Z \cup D_2,
\end{equation}
where $Z$ is foliated by smooth embedded two-spheres $\{\Sigma_t\}_{t\in [-\eps,\eps]}$ with
\begin{equation}
|\Sigma_t|\leq |\Sigma_0| - \frac{\lambda}{4} t^2,
\end{equation}
in particular $|\Sigma_t|< |\Sigma_0|$ for $t\neq 0$,
and where $D_1$ and $D_2$ are smooth embedded three-discs with mean convex boundary.

Theorem \ref{main_foliation_theorem} provides smooth foliations $\{\hat{\Sigma}^i_t\}_{t\in [0,1]}$ of the the discs $D_i$ ($i=1,2$) by a smooth family of mean convex embedded two-spheres, which in particular satisfies
\begin{equation}
|\hat{\Sigma}^i_t|\leq |\partial D_i|,
\end{equation}
$\hat{\Sigma}^i_0=\partial D_i$, and for $t\to 1$ terminates in one-dimensional graphs. Concatenating these foliations with the foliation of $Z$ we conclude that
\begin{equation}
\Sigma_t'= \left\{\begin{array}{ll}
         \hat{\Sigma}^1_{(-t-\eps)/(1-\eps)}, & \text{for }  -1\leq t\leq -\eps\\
        \Sigma_{t}, & \text{for } -\eps\leq t \leq \eps\\
        \hat{\Sigma}^2_{(t-\eps)/(1-\eps)}, & \text{for } \eps\leq t\leq 1
        \end{array}\right.
\end{equation}
(or to be technically precise, rather a smoothed out concatenation of that) gives the desired optimal foliation of $M$ by two-spheres.
\end{proof}

\begin{remark}
An optimal sweep-out can also be produced via min-max methods as follows:
Since $D_{i}$ is mean convex, we can then consider sweep-outs of $D_i$ by two-spheres and consider the min-max problem for manifolds with mean convex boundary.  If the width $W(D_i)$ of $D_i$ was greater than $|\partial D_i|$ then we could infer (c.f. Theorem 2.1 in \cite{MN2}) that $D_i$ contains in its interior a minimal two-sphere, necessarily disjoint from $\Sigma$. This contradicts the fact that in a three-sphere without stable minimal two-spheres any two embedded minimal two-spheres must intersect (see Lemma \ref{mustintersect}).  It follows that $W(D_i)\leq |\partial D_i|$.   Thus we can concatenate a nearly optimal sweep-outs of $D_i$ (that increases area above $W(D_i)$ by less than $|\Sigma_0|-|\Sigma_{\pm\epsilon}|)$ with the cylindrical family on $Z$, to obtain an optimal sweep-out of $M$. However, it can happen that different slices of this optimal sweep-out intersect each other, i.e. the min-max approach in general only produces an optimal sweep-out, but not an optimal foliation.
\end{remark}

Let us end this section with the following standard lemma, which has been used in the above remark. The lemma will also be used in Section \ref{sec_main}, and can also be used to check that the spheres in Theorem \ref{main} indeed have the intersection properties as stated in the introduction.

\begin{lemma}\label{mustintersect}
Let $M$ be a three-sphere containing no stable embedded minimal two-spheres.  Then any two minimal embedded two-spheres in $M$ intersect.
\end{lemma}
\begin{proof}
Let $\Sigma_1$ and $\Sigma_2$ denote two embedded minimal two-spheres, necessarily unstable.  Assume they are disjoint from each other.  One can push off $\Sigma_1$ using the lowest eigenfunction of the stability operator in the direction of $\Sigma_2$, and then consider the level set flow from this mean convex initial condition. Since $M$ contains no stable two-spheres, by the work of White \cite{White_topology,White_size} the flow becomes extinct in finite time and produces a singular foliation of the region enclosed by the initial condition. On the other hand, by the avoidance principle the flow must stay disjoint from $\Sigma_2$; this gives the desired contradiction.
\end{proof}

\section{Two parameter families}\label{sec_twopar}
In this section, we apply the catenoid estimate \cite{KMN} together with the results of the previous sections to show:
\begin{theorem}\label{twoparam}
Let $M$ be a Riemannian $3$-sphere not containing any stable embedded minimal spheres, and let $\{\Sigma_t\}_{t\in [-1,1]}$ be an optimal foliation of $M$ with the 1-width $\omega_1(M)$ realized by the minimal two-sphere $\Sigma_0$.  Then there exists a two parameter sweep-out $\{\Gamma_x\}_{x\in\mathbb{RP}^2}$ of $M$ detecting $\alpha^2$ and with the property that
\begin{equation}\label{totalarea}
\sup_{x\in\mathbb{RP}^2}|\Gamma_{x}| < 2|\Sigma_0|.
\end{equation}
In particular
\begin{equation}
\omega_2(M)<2\omega_1(M).
\end{equation}
\end{theorem}
\begin{proof}

It follows from the definition of optimal foliation (see Section \ref{sec_optimal}) that there exists a $c>0$ so that
\begin{equation}\label{away2}
|\Sigma_t| \leq |\Sigma_{0}|- ct^2
\end{equation}
at first for $|t|$ small, but after decreasing $c$, for all $t\in [-1,1]$.

By following two points on the foliation via normal motion we get two arcs $\alpha:[-1,1]\to M$ intersecting each surface $\Sigma_t$ precisely once, and  disjoint for $t\in (-1,1)$.  There exists a retraction map 
\begin{equation}
R:(M\setminus\alpha) \times [0,1]\rightarrow M\setminus\alpha
\end{equation}
\noindent
satisfying the following properties (where $\Sigma_t':=\Sigma_t\setminus \alpha(t)$):
\begin{enumerate}
\item $R(\Sigma_t', \mu)\subset \Sigma_t'$ for all $t$ and $\mu$
\item $R(\Sigma_t',\mu_2)\subset R(\Sigma_t',\mu_1)$ for all $\mu_2\geq\mu_1$ and $t$
\item $R(a,0)=a$ for all $a\in M\setminus\alpha$
\item $R(\Sigma_t',1)=\beta(t)$
\item for $\mu$ small enough (depending on $t\neq \pm 1$), $R(\Sigma_t',\mu)$ is a disc-type surface with boundary $\{c\;|\; \mbox{dist}_{\Sigma_t'}(\alpha(t),c)=\mu\}$
\end{enumerate}

Let us also define the annuli
\begin{equation}
C(s,t,\mu)= \bigcup_{r\in [s,t]}\partial R(\Sigma_r',\mu),
\end{equation}
where $\partial R(\Sigma_r',\mu)$ is the (circle) boundary of $R(\Sigma_r',\mu)$.

Note that by continuity, for all $a>0$ there exists $b_1(a)>0$ so that whenever $|s-t|\leq b_1(a)$  then 
\begin{equation}\label{b1choice}
|C(s,t,\mu)| < a \mbox{  for any } \mu .
\end{equation}

Moreover, for all $a>0$ there exists $b_2(a)>0$ so that whenever $\mu< b_2(a)$ for any $s,t\in [-1,1]$ there holds
\begin{equation}
|C(s,t,\mu)|\leq |C(-1,1,\mu)|<a.
\end{equation} 

Finally we can define first the sets
\begin{equation}
\Gamma_{s,t}^\vee:= \Sigma_s\cup\Sigma_t,
\end{equation}
\noindent
and then the surfaces
\begin{equation}
\Gamma'_{s,t}:= R(\Sigma_s',\phi(s,t))\cup R(\Sigma_t',\phi(s,t))\cup C(s,t,\phi(s,t)),
\end{equation}
\noindent
where the function $0\leq\phi(s,t)\leq 1$ will be specified in the course of the proof.  
Roughly speaking, the set $\Gamma'_{s,t}$ consists of $\Sigma_s'$ and $\Sigma_t'$ retracted a bit and joined by the annular ``neck" $C(s,t,\phi(s,t))$.
We will open up the neck near the diagonal. In particular, we will chose $\phi(s,t)=1$ for $s=t$, i.e. $\Gamma'_{s,t}$ will be a trivial surface for $s=t$.

Let $\mu>0$ and set $\eps := \min(\mu/2,b_1(C\mu^2/8))$.   Let us define the $\eps$-diagonal
\begin{equation}
D_\eps :=\{(s,t)\in [-1,1]^2 \;:\; |t-s|\leq \eps\} .
\end{equation}

Let us set $\phi(s,t)$ to be the following on $D_\eps$:
\begin{equation*}
    \phi(s,t)|_{D_\eps} = \begin{cases}
 		 1-(s-t)/\eps              & t\leq s \\   
    
                  1-(t-s)/\eps               & s\leq t \\
              
           \end{cases}
\end{equation*}

The function $\phi(s,t)$ on $[-1,1]^2\setminus D_\eps$ can be any non-negative function always less than $b_2(C\eps^2/8)$ that vanishes on $\partial([-1,1]^2)\setminus D_\eps$.  Here, $\phi<b_2$ determines the thickness of the neck in $\Gamma'_{s,t}$ between $\Sigma_s$ and $\Sigma_t$ where the necks are very thin, so the precise form of the function is not important in that region. The most delicate region is near $(s,t)\approx (0,0)$, and in fact we will use the catenoid estimate in the region 
\begin{equation}
B_{\eps,\mu}:=\{(s,t)\in D_\eps: |s|\leq \mu\}.
\end{equation}

Arguing region by region, let us first estimate the area of $\Gamma'_{s,t}$ for $(s,t)\in [-1,1]^2\setminus D_\eps$.
For $(s,t)\in [-1,1]^2\setminus D_\eps$ it follows that either $|t|$ or $|s|$ is at least $\eps/2$ (since by definition on this set $|t-s|>\eps$).   Thus from \eqref{away2} we obtain
\begin{equation}\label{away}
|\Gamma_{s,t}^\vee| \leq 2|\Sigma_{0}| - C\eps^2/4.
\end{equation}

Since we choose $\phi(s,t) < b_2(C\eps^2/8)$ on $[-1,1]^2\setminus D_\eps$ we have 
\begin{equation}\label{collar}
|C(s,t,\phi(s,t))| \leq C\eps^2/8.
\end{equation}

Since $\Gamma'_{s,t}$ is obtained from $\Gamma_{s,t}^\vee$ by removing disks and adding the annular neck $C(s,t,\phi(s,t))$, combining \eqref{away} and \eqref{collar} we infer that
\begin{equation}\label{away}
|\Gamma'_{s,t}| \leq 2|\Sigma_{0}| - C\epsilon^2/8
\end{equation}
for $(s,t)\in [-1,1]^2\setminus D_\eps$.

Let us now estimate the areas of $\Gamma'_{s,t}$ for $(s,t)\in D_\eps\setminus B_{\eps,\mu}$. Since $|s|>\mu$ using the Taylor expansion \eqref{away2} we obtain
\begin{equation}\label{nohole}
|\Gamma_{s,t}^\vee| \leq 2|\Sigma_0| - C\mu^2/4.
\end{equation}
Using that $\eps\leq b_1(C\mu^2/8)$, and arguing as above, we infer that
\begin{equation}\label{defamount}
|\Gamma'_{s,t}| \leq 2|\Sigma_0| - C\mu^2/8
\end{equation}
for $(s,t)\in D_\epsilon\setminus B_{\epsilon,\mu}$.

So far we have established that $|\Gamma'_{s,t}| < 2|\Sigma_0|$ for $(s,t)\in [-1,1]^2\setminus B_{\eps,\mu}$.  It remains to use the catenoid estimate to adjust the family $\Gamma'_{s,t}$ in the region $B_{\eps,\mu}$ so that the areas of surfaces corresponding to this region in parameter-space are also below $2|\Sigma_0|$.  Note that everything above remains valid if we need to shrink $\eps$ further, which applying the catenoid estimate may require.

Let us further divide the region $B_{\eps,\mu}$ into its part ``above" the diagonal, and its part ``below."  Namely, set
\begin{equation}
B^+_{\eps,\mu}:=B_{\eps,\mu}\cap \{s\leq t\}
\end{equation}
and
\begin{equation}
B^-_{\eps,\mu}:=B_{\eps,\mu}\cap \{t\leq s\}.
\end{equation}

We will handle $B^-_{\eps,\mu}$ as the argument works mutatis mutandis for $B^+_{\eps,\mu}$.  

Set $r(x)=\mbox{dist}_{\Sigma_0}(x,\alpha(0))$.  Let us define the following logarithmic cutoff function:
\begin{equation*}
    \eta_\tau(x) = \begin{cases}
 		 1              & r(x)\geq \tau \\   
    
                  (1/\log(\tau))(\log \tau^2 - \log r(x))               & \tau^2\leq r(x)\leq \tau \\
                  0& r(x)\leq \tau^2
           \end{cases}
\end{equation*}
Let $\phi$ denote the normalized lowest eigenfunction of the stability operator so that $\Sigma_t=\exp_{p\in\Sigma_0}(t\phi(p) \nu(p))$ for $|t|<2\mu$, where we now choose the constant $\mu$ small enough.

Assume $s\in[-\mu,-\mu]$.  Let us denote the surface
\begin{equation}
\Lambda'_{s,\tau}:=\{\exp_p((s-\eps+\eps\eta_\tau(p))\phi(p)\nu(p))\;|\; p\in \Sigma_0\setminus B_{\tau^2}(\alpha(0))\}
\end{equation}
and the surface
\begin{equation}
\Lambda''_{s-\eps,\tau}:=\{\exp_p((s-\eps)\phi(p)\nu(p))\;|\; p\in \Sigma_0\setminus B_{\tau^2}(\alpha(0))\}.
\end{equation}
Finally for $s\in[-\mu,-\mu]$) denote the closed surface
\begin{equation}
\Lambda_{s,\tau}:=\Lambda'_{s,\tau}\cup \Lambda''_{s-\eps,\tau}.
\end{equation}
Loosely speaking, $\Lambda_{s,\tau}$ consists of $\Sigma_{s-\eps}$ together with $\Sigma_s$ with one disk removed from each and in their place a ``logarithmic" neck obtained from the cutoff function $\eta_\tau(x)$ that joins the two layers.

It follows from Theorem 2.4 in \cite{KMN} (the Catenoid Estimate) that if we shrink $\eps$ enough, there exists $\bar{\tau}>0$ so that for all $\tau\in [0,\bar{\tau}]$, and $s\in [-\mu,\mu]$ we can estimate the area of $\Lambda_{s,\tau}$ as:
\begin{equation}\label{cat}
|\Lambda_{s,\tau}|\leq 2|\Sigma_0|-C\eps^2.
\end{equation}
and
\begin{equation}\label{cat2}
|\Lambda_{s,\bar{\tau}}|\leq 2|\Sigma_0|-C\bar{\tau}^2.
\end{equation}
The only difference in establishing \eqref{cat} and \eqref{cat2} from the setting of Theorem 2.4 in \cite{KMN} is that here the components of $\Lambda_{s,\tau}$ are not symmetrically placed about the minimal surface $\Sigma_0$.  The point is that for $s\in[-\mu,\mu]$, the sum of the squared distance of $s-\eps$ from zero and the squared distance of $s$ from zero is at least $\eps^2/2$, and thus upon Taylor expanding the areas for $\Lambda_{s,\tau}$, one still obtains a negative $\eps^2$ term, which overpowers the $\eps^2/(-\log\eps)$ coming from ``opening the neck".   

Thus we can replace the sweep-out $\Gamma'_{s,t}$ in the region 
\begin{equation}
B_{\eps,\mu}^-\cap\{s-\eps\leq t\leq s-\eps/2\}
\end{equation}
by the surface $\Lambda_{s,(2\bar{\tau}/\eps)(t-(s-\eps))}$.

Using the globally defined retractions $R$, (c.f. lines 2.58-2.65 in \cite{KMN}) one can replace the sweep-out $\Gamma'_{s,t}$ in 
\begin{equation}
B_{\eps,\mu}^-\cap\{s-\eps/2\leq t\leq s\}
\end{equation}
while still maintaining the bound \eqref{cat} so that $\Gamma'_{s,s}$ is a trivial surface consisting of the point $\beta(s)$ for each $s\in [-\mu,\mu]$.  This completes the construction of the required amendments to $\Gamma'_{s,t}$ in the region $B^-_{\eps,\mu}$, and thus also in $B_{\eps,\mu}$.

Note that we have adjusted the sweep-out $\Gamma'_{s,t}$ in $B_{\eps,\mu}$ and thus the resulting sweep-out is not continuous over the two intervals $\partial B_{\eps,\mu}\cap\{s=\pm\mu\}$ (it is continuous over the other parts of $\partial B_{\eps,\mu}$ since we have chosen that the neck-size $\phi(s,t)=0$ on these intervals).
But because of \eqref{nohole}, by shrinking $\eps$ if necessary it is clear that we can make a transition region between the logarithmically cut-off surfaces in the amended sweep-outs and retracted surfaces while maintaining the area bound \eqref{totalarea}.  Indeed, by \eqref{defamount} area of $\Gamma'_{s,t}$ is less than $2|\Sigma_0|$ by a definite amount independent of $\eps$, and ``opening the neck" logarithmically, linearly using retractions $R$, or any interpolation between these two functions only adds area of order $\eps^2$.  
This completes the construction of the family $\Gamma_{s,t}$ satisfying the area bound \eqref{totalarea}.
\\
\\
It remains to show that we get a two-parameter sweep-out which detects $\alpha^2$.  Note that because we are using $\mathbb{Z}_2$ coefficients, the sweep-out $\Gamma_{s,t}$ is invariant under flipping through the diagonal:
\begin{equation}\label{identification}
\Gamma_{s,t}=\Gamma_{t,s}.
\end{equation}

Making the identifications of \eqref{identification}, we obtain a sweep-out parameterized by a triangle in $[-1,1]^2$, which after appropriate boundary identifications gives a sweep-out parameterized by $\mathbb{RP}^2$.  We will now show that this sweep-out $\{\Gamma_x\}_{x\in\mathbb{RP}^2}$ detects $\alpha^2$.   
To this end, let $\lambda$ represent a generator of $H^1(\mathbb{RP}^2,\mathbb{Z}_2)$ so that $\lambda^2\neq 0$.  As $\pi_1(\mathbb{RP}^2)=\mathbb{Z}_2$, consider the curve $\gamma$ in $\mathbb{RP}^2$ given by $s\rightarrow (s,-1)$ for $s\in[-1,1]$ which represents the only non-trivial element of the fundamental group of $\mathbb{RP}^2$.   The surfaces $\{\Gamma_{s,-1}\}_{s\in [-1,1]}$ corresponding to $\gamma$ give a $1$-sweep-out of $M$ as they are equal to the optimal foliation $\{\Sigma_s\}_{s\in [-1,1]}$.  Moreover, $\lambda(\gamma)=1$.  On other hand, any homotopically trivial closed curve $\gamma\subset\mathbb{RP}^2$ retracts to a point and thus $\lambda(\gamma)=0$.  
Thus by the criterion in Section \ref{sec_topo}, the sweep-out $\{\Gamma_x\}_{x\in\mathbb{RP}^2}$ detects $\alpha^2$.
\end{proof}

\section{Proof of Theorem \ref{main}}\label{sec_main}

In this section, we prove Theorem \ref{main}, which we restate.
\begin{theorem}[Existence of minimal two-spheres] 
Let $M$ be a three-manifold diffeomorphic to $\mathbb{S}^3$ and endowed with a bumpy metric.
Then $M$ contains at least $2$ embedded minimal two-spheres. More precisely, exactly one of the following alternatives holds:
\begin{enumerate}
\item $M$ contains at least $1$ stable embedded minimal two-sphere, and at least $2$ embedded minimal two-spheres of index one.
\item $M$ contains no stable embedded minimal two-sphere, at least $1$ embedded minimal two-sphere $\Sigma_1$ of index one, and at least $1$ embedded minimal two-sphere $\Sigma_2$ of index one or two.  In this case
\begin{equation}
|\Sigma_1|<|\Sigma_2|\ < 2|\Sigma_1|.
\end{equation}
\end{enumerate}
\end{theorem}

\begin{proof}
Suppose first $M$ contains an embedded stable minimal two-sphere $\Sigma$.  Then by \cite{KL}, we obtain that each of the two balls in $M$ bounded by $\Sigma$ contains in its interior an index $1$ minimal sphere.   Thus in this case $M$ contains three minimal two-spheres, more precisely 1 stable and 2 index one spheres, and thus case (1) is satisfied.

Suppose now that $M$ contains no embedded stable minimal two-spheres.  Let us apply the Min-Max Theorem \ref{minmax} to the families detecting $\alpha$ and $\alpha^2$ to obtain stationary integral varifolds $V_1$ and $V_2$, which are collections of pairwise disjoint embedded minimal two-spheres, possibly with multiplicity.  Since by Lemma \ref{mustintersect} any two embedded minimal two-spheres intersect, it follows that $V_1$ and $V_2$ are in fact connected minimal two-spheres $\Sigma_1$, $\Sigma_2$ occurring with multiplicity $k_1$ and $k_2$, respectively.  

By Lemma \ref{lemma_width}  it follows that $k_1=1$ and thus that $|\Sigma_1|=\omega_1(M)$.  Moreover, $\Sigma_1$ is a smallest area minimal two-sphere in $M$, and its index is one.  

By Theorem \ref{twoparam}, we obtain
\begin{equation}\label{widthbound2}
k_2 |\Sigma_2|=\omega_2(M)  <  2\omega_1(M) = 2 |\Sigma_1|.
\end{equation}
Combining this with the property that $\Sigma_1$ is a smallest area minimal two-sphere, and thus in particular $|\Sigma_1|\leq |\Sigma_2|$, we infer that $k_2=1$.

Suppose first that $\omega_2(M) > \omega_1(M)$. This implies in particular that $\Sigma_1\neq \Sigma_2$. By \eqref{indexbound} it follows that the index of $\Sigma_2$ is at most $2$.  This completes the proof in the case that $\omega_2(M)>\omega_1(M)$.




Suppose instead that $\omega_2(M) =\omega_1(M)$.  By Lusternik-Schnirelmann theory (see Theorem \ref{ls} below) $M$ contains infinitely many embedded minimal two-spheres of area $\omega_1(M)$.  But this infinite family of two-spheres, as they have uniformly bounded areas, must have a convergent subsequence.  This results in a Jacobi field, violating the bumpiness of the metric.  It follows that this case cannot occur.  
 
Thus case (2) is established when $M$ contains no stable embedded minimal two-spheres. This completes the proof. 
\end{proof}

It remains to prove that (in the notation of the above proof)

\begin{theorem}\label{ls}
If $\omega_2(M)=\omega_1(M)$, then $M$ contains infinitely many embedded minimal two-spheres of area $\omega_1(M)$.
\end{theorem}

Theorem \ref{ls} is well-known to experts (c.f. Theorem 6.1 in \cite{MN}) but we include a brief proof for the reader's convenience. The following vanishing lemma is the essential topological ingredient needed in the proof (for a short proof of the lemma, see the Appendix of \cite{G}):

\begin{lemma}[{Lusternik-Schnirelmann vanishing lemma}]\label{vanishing}
Let $X$ be a CW complex and $A$ and $B$ be open subsets of $X$.  Let $\alpha,\beta\in H^*(X,R)$ be cohomology classes, where $R$ is any ring.  If $\alpha |_{A}=0$ and $\beta |_{B}=0$, then $\alpha\cup\beta$ vanishes on $A\cup B$.
\end{lemma}

\begin{proof}[{Proof of Theorem \ref{ls}}]
Suppose $M$ contains only finitely many minimal embedded two spheres of area $\omega_1(M)$.   Denote by $\mathcal{S}_{\min}$ the space of such two-spheres.

Consider a minimizing sequence of sweep-outs $\Phi_i$ defined on $X_i$ the area of whose maximal slice is approaching $\omega_2(M)=\omega_1(M)$.    Thus we have
\begin{equation}\label{minimizing}
\sup_{x\in X_i} |\Phi_i(x)|\leq \omega_1(M)+\epsilon_i,
\end{equation}
where $\epsilon_i\rightarrow 0$.

Let $\eps>0$. For each $i$, consider the following subsets of $X_i$:
\begin{equation}
A_i:= \{x\in X_i \;|\; \mathcal{F}(\Phi_i(x),\mathcal{S}_{\min})<\eps\}
\end{equation}
and
\begin{equation}
B_i:= \{x\in X_i \;|\; \mathcal{F}(\Phi_i(x),\mathcal{S}_{\min})>\eps/2\}.
\end{equation}

Since by assumption $A_i$ is a neighborhood in the $\mathcal{F}$-metric of finitely many minimal surfaces, it follows that $\alpha |_{A_i}=0$ provided $\eps$ is chosen small enough, c.f. \cite{MN2}.

If $\alpha |_{B_i}=0$, then since $X_i=A_i\cup B_i$, it follows from Lemma \ref{vanishing} that $\alpha\cup\alpha |_{X_i}$ also vanishes.  This contradicts the fact that the families $\Phi_i$ detect $\alpha\cup\alpha$ by assumption.
Thus $\alpha |_{B_i}$ does not vanish, and thus the families $\Phi_i |_{B_i}$ give non-trivial sweep-outs of $M$.   

Consider now the min-max width
\begin{equation}
W=\inf_i \sup_{x\in B_i} |\Phi_i(x)|.
\end{equation}

On the one hand, from \eqref{minimizing} we have
\begin{equation} \label{we}
W\leq\omega_1(M). 
\end{equation}

On the other hand, since $\Phi_i |_{B_i}$ detects $\alpha$, we obtain
\begin{equation}
\omega_1(M)\leq W,
\end{equation}
and thus
\begin{equation}\label{isequal}
W=\omega_1(M).
\end{equation}

However, since the families $\Phi_i |_{B_i}$ are by definition of $B_i$ a definite distance away from the space of embedded minimal two-spheres of area $\omega_1(M)$, and $\omega_2(M)=\omega_1(M)$, it follows that the family  $\Phi_i |_{B_i}$ contains no almost minimizing subsequence.   Thus the ``pull-tight" argument of Almgren-Pitts allows one to construct a homotopic family $\Phi'_i |_{B_i}$ with all areas strictly below $\omega_1(M)$.  This contradicts \eqref{isequal}, and finishes the proof.
\end{proof}

\section{Smooth mean convex foliations in Euclidean space}\label{sec_fol_eucl}
 
The goal of this section is to prove:

\begin{theorem}\label{thm_foliation_r3}
Let $K\subset \mathbb{R}^3$ be a smooth 3-disc with mean convex boundary. Then there exists a smooth foliation $\{ \Sigma_t \}_{t\in [0,1]}$ of $K$ by mean convex embedded 2-spheres.
\end{theorem}

To prove Theorem  \ref{thm_foliation_r3} we will follow the approach from Buzano-Haslhofer-Hershkovits \cite{BHH}. In that prior work,
using a variant of mean curvature flow with surgery where neck regions are deformed to tiny strings instead of being cut out completely, it has been shown in particular that there exists a mean convex isotopy $\{K_t\}_{t\in [0,T_1]}$ that deforms $K$ into a so-called marble-tree (see \cite[Def. 5.1]{BHH}). Unfortunately, in most situations the isotopy constructed in \cite{BHH} is not monotone. We will now modify the approach to construct a mean convex isotopy $\{K'_t\}_{t\in [0,T_1]}$ that deforms $K$ into a marble-tree and has the additional good property that it is strictly monotone, i.e.
\begin{equation}
t_2 > t_1\quad \Rightarrow \quad K'_{t_2}\subset \textrm{Int}(K'_{t_1}).
\end{equation}
Once this is achieved, the family $\Sigma_t=\partial K'_t$ (concatenated with a final bit which shrinks the marble-tree to its skeleton) will give the desired smooth mean convex foliation of $K$.

\begin{proof}[Proof of Theorem \ref{thm_foliation_r3}]
Given the initial domain $K\subset \mathbb{R}^3$ we choose all the parameters for the following argument suitably as in \cite[p. 25]{BHH}.

Consider the evolution $\{K_t\}$ by mean curvature flow with surgery given by \cite[Thm. 1.21]{HK} with initial condition $K$.
Let $0 < t_1 < \ldots < t_\ell$ be the times where there is some surgery and/or discarding (see \cite[Def. 1.17]{HK}). By the definition of a flow with surgery at each $t_i$ there are finitely many (possibly zero) $\delta$-necks with center $p_i^j$ and radius $r_{\textrm{neck}} = H_{\textrm{neck}}^{-1}$ that are replaced by a pair of standard caps \cite[Def. 2.4]{HK}. Similarly, for each discarded component $C_i^j$ (see \cite[Def. 1.17]{HK}) \cite[Prop. 7.4]{BHH} gives a finite collection of $\varepsilon$-neck points  whose centers and radii are denoted by $p_i^{jk}$ and $r_i^{jk}$. Write $B_i^j:=B_{10\Gamma_{\textrm{cap}} r_{\textrm{neck}}}(p_i^j)$ and $B_i^{jk}:=B_{10\Gamma_{\textrm{cap}} r_i^{jk}}(p_i^{jk})$, where $\Gamma_{\textrm{cap}}$ is the cap separation parameter for the surgeries (see \cite[Def. 2.4]{HK}). These balls are disjoint and the isotopy $\{K_t\}_{t\in [0,T_1]}$ constructed in \cite[Sec. 9]{BHH} is monotone outside of
\begin{equation}\label{def_of_x}
X:=\bigcup_{i,j} B_i^j \cup \bigcup_{i,j,k}B_i^{jk},
\end{equation}
i.e. we have that
\begin{equation}
t_2 \geq t_1\quad \Rightarrow \quad K_{t_2}\setminus X \subseteq K_{t_1}\setminus X.
\end{equation}
We will now modify the construction to make the isotopy monotone also inside $X$, and in fact strictly monotone everywhere.

\begin{proposition}[{Replacement for \cite[Prop. 7.4]{BHH}}]\label{prop_replacement2}
For $\varepsilon$ small enough, every capped $\varepsilon$-tube (see \cite[Def. 7.3]{BHH}) is mean convex isotopic to a marble tree, and the isotopy can be chosen strictly monotone.
\end{proposition}

\begin{proof}[{Proof of Proposition \ref{prop_replacement2}}]
Let $(K,\gamma)$ be an $\varepsilon$-neck, with $\varepsilon$ small enough so that the following argument works.

Our argument starts similarly as in the proof of \cite[Prop. 7.4]{BHH}. Namely, let $p_\pm\in\gamma$ be $\varepsilon$-neck points that are as close as possible to the endpoints $\bar{p}_\pm$, respectively. Let $\mathcal{I}\subset \gamma$ be a maximal collection of $\varepsilon$-neck points with $p_\pm\in \mathcal{I}$ such that for any pair $p,q\in \mathcal{I}$ the separation between the points is at least $100 \max\{\varepsilon^{-1}, \Gamma_{\textrm{cap}}\} \max\{H(p)^{-1}, H(q)^{-1}\}$.
For each $p\in\mathcal{I}$ we replace the $\varepsilon$-neck with center $p$ by a suitable pair of opposing standard caps as in \cite[Def. 6.2]{BHH}.\footnote{We remark that it is also possible to give an alternative proof of Proposition  \ref{prop_replacement2} without performing surgeries. However, the construction of suitable isotopies through the surgery necks is needed for the next corollary.} Denote the post-surgery domain by $K^\sharp$. Let $\tilde{\gamma}$ be the disjoint union of almost straight curves connecting the opposing standard caps as in \cite[Lem. 6.4]{BHH}. Note that these curves are a perturbation of pieces of $\gamma$.

Let $\mathcal{G}_{r_s}$ be the gluing map from \cite[Thm. 4.1]{BHH}, with small enough string radius $r_s$. Let $K^\sharp_t$ be the mean curvature flow evolution of $K^\sharp$, and let $\{\tilde{\gamma}_t\}$ be the family of curves which follows $K^\sharp_t$ by normal motion starting at $\tilde{\gamma}_0=\tilde{\gamma}$. We claim that for $\bar{t}$ small enough, and for a suitable family of curves $\{\tilde{\gamma}'_t\}$ very close to $\{\tilde{\gamma}_t\}$, there exists a strictly monotone mean convex isotopy between $K$ and $\mathcal{G}_{r_s}(K^\sharp_{\bar{t}},\tilde{\gamma}'_{\bar{t}})$.

To see this we fix a partition $0<\bar{t}_1<\bar{t}_2<\bar{t}_3<\bar{t}$, where $\bar{t}$ is small enough, and construct a suitable isotopy step by step as follows.

Note first that if there is a surgery with center $p$, then for $t$ small enough $K^\sharp_t$ can be expressed locally as a graph in $B_{6\Gamma_{\textrm{cap}}r_{\textrm{neck}}}(p)$ with small $C^{20}$-norm over a pair of opposing evolving standard caps $K^{\textrm{st}}_t$ starting at distance $\Gamma_\textrm{cap}r_{\textrm{neck}}$ from $p$ (c.f. \cite[Def. 2.4]{HK}).

Using the graphical representation we can project to these standard caps while at the same time letting evolve the configuration a little bit to ensure that the deformation is strictly monotone. Choosing $\bar{t}_1$ small and then recalling that $\eps$ is very small, we can thus find a strictly monotone mean convex isotopy $\{L_t\}_{t\in [0,\bar{t}_1]}$ starting at $L_0=K^\sharp$, such that for each surgery point $p$ we have that $L_{\bar{t}_1}\cap B_{5\Gamma_{\textrm{cap}}r_{\textrm{neck}}}(p)$ is exactly a pair of opposing standard caps. Moreover, we can slightly perturb the family $\{\tilde{\gamma}_t\}$ to get a family $\{\tilde{\gamma}'_t\}$ with the property that 
$\tilde{\gamma}'_{\bar{t}_1}$ connects these opposing standard caps in exactly straight lines.

Next, using \cite[Prop. 3.12]{BHH} we can find a strictly monotone mean convex isotopy $\{L_t\}_{t\in [\bar{t}_1,\bar{t}_2]}$, such that at time $\bar{t}_2$ for each surgery point $p$ we have that $L_{\bar{t}_2}\cap B_{5\Gamma_{\textrm{cap}}r_{\textrm{neck}}}(p)$ is pair of standard caps connected by a neck of radius $\rho(0.98)r_{\textrm{neck}}$, where $\rho$ is the radius function from gluing the ball and the cylinder \cite[Prop. 4.2]{BHH}.

Third, taking into account property (a) of the standard cap from \cite[Def. 2.8]{BHH} and using the fact that the gluing in the rotationally symmetric case is described by the explicit model from  \cite[Prop. 4.2]{BHH}, we can now decrease the neck radius from $\rho(0.98)r_{\textrm{neck}}$ down to $2r_s$ via a strictly monotone mean convex isotopy $\{L_t\}_{t\in [\bar{t}_2,\bar{t}_3]}$.

Finally, interpolating again between the rotationally symmetric and non-symmetric situation via the graphical representation as above we can find a strictly monotone mean convex isotopy $\{L_t\}_{t\in [\bar{t}_3,\bar{t}]}$ with $L_{\bar{t}}= \mathcal{G}_{r_s}(K^\sharp_{\bar{t}},\tilde{\gamma}'_{\bar{t}})$. During the process we shrink the string radius from $2r_s$ to $r_s$ to ensure that the deformation is strictly monotone.

It remains to construct a strictly monotone mean convex isotopy that deforms $\mathcal{G}_{r_s}(K^\sharp_{\bar{t}},\tilde{\gamma}'_{\bar{t}})$ into a marble tree. To this end, if we choose $\bar{t}$ very small, then $K^\sharp_{\bar{t}}$ is as close as we want to $K^\sharp$. Then, inferring as in the proof of \cite[Prop. 7.4]{BHH} that the connected components of $K^\sharp_{\bar{t}}$ are either convex or capped off cylinders we see that there exists a strictly monotone mean convex isotopy $\{L_t\}_{t\in [0,1]}$ starting at $L_0=K^\sharp_{\bar{t}}$ such that $L_1$ is a finite union of round balls. 
Letting $r_{\textrm{min}}$ be the smallest radius among the radii of the balls of $L_1$, let $\{L_t\}_{t\in [1,2]}$ be a strictly monotone mean convex isotopy that concatenates smoothly at $t = 1$ and shrinks all balls further to balls of radius $r_{\textrm{min}}/10$.
Let $\{\gamma_t\}_{t\in[0,2]}$ be the family of curves which follows $L_t$ by normal motion starting at $\gamma_0=\tilde{\gamma}_{\bar{t}}$. Let $r_s(t)$ be a slowly decreasing positive function starting at $r_s(0)=r_s$ from above.
Then $\{\mathcal{G}_{r_s(t)}(K^\sharp_{\bar{t}},\tilde{\gamma}'_{\bar{t}})\}_{t\in [0,2]}$ is a strictly monotone mean convex isotopy that deforms the domain $\mathcal{G}_{r_s}(K^\sharp_{\bar{t}},\tilde{\gamma}'_{\bar{t}})$ into a marble tree with marble radius $r_{\textrm{min}}/10$ and string radius $r_s(2)$.

Taking a smooth concatenation of the above isotopies, this proves the proposition.
\end{proof}

\begin{corollary}[{Replacement for \cite[Prop. 6.5]{BHH}}]\label{prop_replacement1}
There exists a constant $\bar{\delta}>0$ with the following significance. Let $\delta\leq\bar{\delta}$, assume $K^\sharp$ is obtained from $K^-$ by performing surgeries on a disjoint collection of $\delta$-necks, and let $\gamma$ be the union of the almost straight lines connecting the tips of the opposing standard caps as in \cite[Lem. 6.4]{BHH}. Let $\{K^\sharp_t\}$ be a strictly monotone mean convex evolution of $K^\sharp$, and let $\{\gamma_t\}$ be the family of curves which follows $K^\sharp_t$ by normal motion starting at $\gamma$.
Then, for $r_s$ small enough, every small enough $\bar{t}$, and a suitable perturbation of $\{\gamma_t\}$ which we denote again by $\{\gamma_t\}$, there exists a strictly monotone mean convex isotopy between $K^-$ and $\mathcal{G}_{r_s}(K^\sharp_{\bar{t}},\gamma_{\bar{t}})$.
\end{corollary}

\begin{proof}[{Proof of Corollary \ref{prop_replacement1}}]
This follows from the proof of Proposition \ref{prop_replacement2}.
\end{proof}

Continuing the proof of Theorem \ref{thm_foliation_r3}, let us first construct suitable isotopies for the discarded components $C_i^j$ at time $t_i$.

\begin{claim}\label{claim_inductionhyp}
All discarded components $C_i^j$ are isotopic to a marble tree, with an isotopy which strictly monotone and mean convex.
\end{claim}

\begin{proof}[{Proof of Claim \ref{claim_inductionhyp}}]
Our topological assumption on the initial domain $K$ together with the nature of the surgery process (see \cite[Def. 1.17]{HK}) implies that all discarded components are diffeomorphic to balls. Thus, by Corollary \cite[Cor. 8.9]{BHH}, each discarded component is either (a) convex or (b) a capped $\eps$-tube. Contracting to a small ball in case (a), respectively using Proposition \ref{prop_replacement2} in case (b), we can find a strictly monotone mean convex isotopy to a marble tree.
\end{proof}

Let $\A_i$ be the assertion that each connected component of the pre-surgery domain $K^i:=K_{t_{i}}^-$ is isotopic to a marble tree, with an isotopy which is strictly monotone and mean convex. Since $K_{t_{\ell}}^+=\emptyset$, we see that at the final time $t_{\ell}$ there was only discarding and no replacement of necks by caps (see \cite[Def. 1.17]{HK}). Thus, all connected components of $K^\ell=K_{t_\ell}^-$ get discarded, and Claim \ref{claim_inductionhyp} shows that $\A_\ell$ holds.

\begin{claim}\label{claim_inductionstep}
If $0<i<\ell$ and $\A_{i+1}$ holds, so does $\A_i$.
\end{claim}

\begin{proof}[{Proof of Claim \ref{claim_inductionstep}}]
Smooth evolution by mean curvature flow provides a strictly monotone mean convex  isotopy between $K_{t_i}^+$ and $K^{i+1}$. Recall that $K_{t_{i}}^+\subseteq K_{t_{i}}^\sharp\subseteq K_{t_{i}}^-=K^i$ is obtained by performing surgery on a minimal collection of disjoint $\delta$-necks separating the thick part and the trigger part and/or discarding connected components that are entirely covered by canonical neighborhoods (see \cite[Def. 1.17]{HK}).

By induction hypothesis the connected components of $K^{i+1}$ are isotopic to marble trees, and by Claim \ref{claim_inductionhyp} the discarded components are isotopic to marble trees as well, where all the isotopies can be chosen strictly monotone and mean convex. It follows that there exists a strictly monotonic mean convex isotopy $\{L_t\}_{t\in [0,1]}$ deforming $L_0=K_{t_{i}}^\sharp$ into a union of marble trees $L_1$. If $L_0$ has more than one connected component, then we glue together these isotopies as follows.

For each surgery neck at time $t_i$, select an almost straight line $\gamma_i^j$ between the tips of the corresponding pair of standard caps in $K_{t_{i}}^\sharp$ as in \cite[Lem. 6.4]{BHH}. Let $\gamma=\bigcup_{j}\gamma_i^j$. Define $\gamma_t$ by following the points where $\gamma_t$ touches $\partial L_t$ via normal motion.
By Corollary \ref{prop_replacement1} the domain $K^i=K_{t_i}^-$ is isotopic to $\mathcal{G}_{r_s}(L_{\bar{t}}^\sharp,\gamma_{\bar{t}})$ via a strictly monotone mean convex isotopy, provided $r_s$ and $\bar{t}$ are small enough. Finally define $\{\gamma_t\}_{t\in [\bar{t},1]}$ essentially by following the points where $\gamma_t$ touches $\partial L_t$ via normal motion. It can happen at finitely many times $t$ that $\gamma_{t}$ hits $\partial X$, see \eqref{def_of_x}. In that case, we modify $\gamma_t$ according to \cite[Lem. 9.4]{BHH}, and then continue via normal motion. Let $r_s(t)$ be a slowly decreasing positive function starting at $r_s(0)=r_s$ from above. Then $\mathcal{G}_{r_s(t)}(L_t,\gamma_t)_{t\in [\bar{t},1]}$ gives the last bit of the desired isotopy. This finishes the proof of Claim \ref{claim_inductionstep}.
\end{proof}

It follows from backwards induction on $i$, that $\A_1$ holds. Smooth mean curvature flow provides a strictly monotone mean convex isotopy between $K$ and $K^1$. In particular, $K^1$ has only one connected component. Thus there exists a strictly monotone mean convex isotopy deforming $K$ into a marble tree. Finally, we can shrink the marble radius and the string radius to zero to obtain the final bit of the desired foliation. This finishes the proof of Theorem \ref{thm_foliation_r3}.
\end{proof}

\section{Flow with surgery in general ambient manifolds}\label{sec_surgery_ambient}

Mean curvature flow with surgery in general ambient manifolds has been constructed by Brendle-Huisken \cite{BrendleHuisken}. 
For our application however, in addition to existence it is important to also have a canonical neighborhood theorem -- most conveniently in the form of \cite[Thm. 1.22]{HK}. The goal of this section is thus to extend the construction from Haslhofer-Kleiner \cite{HK} to general ambient manifolds.

Let us start by adapting some definitions to the setting of general ambient manifolds.

\begin{definition}[{c.f. {\cite[Def. 1.1]{HK}}}]\label{def_smooth_alpha}
Let $\alpha>0$.
A smooth $\alpha$-Andrews flow $\{K_t\subseteq U\}_{t\in I}$ in an open set $U\subseteq M$ in a closed Riemannian three-manifold $M$ is a smooth family of mean convex domains moving by mean curvature flow with $\inf H\geq 4\alpha/\textrm{inj}(M)$ such that for every $p\in\partial K_t$ the two closed balls $\bar{B}^\pm(p)$ with radius $r(p)=\alpha/H(p)$ and center $c^\pm(p)=\exp_p(\pm r(p)\nu(p))$, where $\nu$ is the inwards unit normal, satisfy
$\bar{B}^+(p)\cap U\subseteq K_t$ and $\bar{B}^-(p)\cap U\subseteq U\setminus\textrm{Int}(K_t)$, respectively.
\end{definition}

The notion of $(\alpha,\delta)$-flow is then defined as in \cite[Def. 1.3]{HK}, where smooth $\alpha$-Andrews flows are now defined via Definition \ref{def_smooth_alpha} and strong $\delta$-necks and surgery on such a neck are defined as follows.

\begin{definition}[{c.f. {\cite[Def. 2.3]{HK}}}]\label{def_strong_neck}
We say that an $(\alpha,\delta)$-flow $\mathcal{K}=\{K_t\subseteq U\}_{t\in I}$ has a strong $\delta$-neck with center $p$ and radius $s$ at time $t_0\in I$ if
$4s/\delta\leq \textrm{inj}(M)$ and
$\{ s^{-1}\cdot \exp_p^{-1}(K_{t_0+s^2 t}\cap B_{2s/\delta}(p)\cap U)\}_{t\in (-1,0]}$ is $\delta$-close in $C^{\lfloor 1/\delta \rfloor}$ in $B_{1/\delta}\cap s^{-1}\cdot \exp_p^{-1}(B_{2s/\delta} \cap U)\subset\mathbb{R}^3$ to the evolution of a solid round cylinder $\bar{D}^2\times \mathbb{R}$ with radius $1$ at $t=0$.
\end{definition}

The notion of replacing the final time slice of a strong $\delta$-neck by a pair of standard caps (surgery) is then defined as in \cite[Def. 2.4]{HK} where the items (3) and (4) of that definition are replaced by:
\begin{itemize}
\item[(3')] If  $B(p,5\Gamma s)\subseteq U$, then for every point $p_\sharp \in \partial K^\sharp\cap B(p,5\Gamma s)$ with $\lambda_1(p_\sharp)<0$, there is a point $p_{-}\in \partial K^-\cap B(p,5\Gamma s)$ with $\tfrac{\lambda_1}{H}(p_{-})\leq \tfrac{\lambda_1}{H}(p_\sharp)+\delta'(s)$.
\item[(4')] $B(p,10\Gamma s)\subseteq U$, then $s^{-1}\cdot \exp_p^{-1}(K^\sharp)$ is $\delta'(\delta)$-close in $B(0,10\Gamma)\subset \mathbb{R}^3$ to a pair of opposing standard caps at distance $\Gamma$ from the origin.
\end{itemize}

In the above, $\delta'$ is a positive function with $\lim_{x\to 0}\delta'(x)=0$.

Having adapted the definitions to the setting of general ambient manifolds, let us now discuss the local curvature estimate, the convexity estimate and the global convergence theorem.

\begin{theorem}[{Local curvature estimate, c.f. \cite[Thm. 1.6]{HK}}]\label{thm_loc_curv}
There exist $\bar{\delta}=\bar{\delta}(\alpha,M)>0$, $\bar{r}=\bar{r}(\alpha,M)>0$, $\rho=\rho(\alpha,M)>0$, and $C_\ell=C_\ell(\alpha,M)<\infty$ with the following property. If $\mathcal{K}$ is an $(\alpha,\delta)$-flow ($\delta\leq \bar{\delta}$) in a parabolic ball $P(p,t,r)\subset M\times \mathbb{R}$ centered at a boundary point $p\in\partial K_t$ with $H(p,t)\leq r^{-1}$ and $r\leq\bar{r}$, then
\begin{equation}
\sup_{P(p,t,\rho r)\cap \partial\mathcal{K}}| \nabla^\ell A | \leq C_\ell r^{-1-\ell}.
\end{equation} 
\end{theorem}

\begin{proof}
Observe that along any contradictory sequence $H(p_j,t_j)\leq \bar{r}_j^{-1}\to \infty$ and $s_j\leq \tfrac{1}{4}\delta_j \textrm{inj}(M)\to 0$. Thus, after blowup everything reduces to the situation in Euclidean space and the proof from \cite{HK} goes through. The only nontrivial modification is to show that Huisken's monotonicity formula holds with arbitrarily small error terms if the ambient space is almost Euclidean, but this has been proved by Hamilton \cite[Thm. B]{Ham_mon_amb}
\end{proof}

\begin{theorem}[{Convexity estimate, c.f. \cite[Thm. 1.8]{HK}}]\label{thm_conv}
For all $\eps>0$, there exist $\bar{\delta}=\bar{\delta}(\alpha,M)>0$, $\bar{r}=\bar{r}(\eps,\alpha,M)>0$, and $\eta=\eta(\eps,\alpha,M)<\infty$ with the following property. If $\mathcal{K}$ is an $(\alpha,\delta)$-flow ($\delta\leq \bar{\delta}$) in a parabolic ball $P(p,t,\eta r)\subset M\times \mathbb{R}$ centered at a boundary point $p\in\partial K_t$ with $H(p,t)\leq r^{-1}$ and $r\leq\bar{r}$, then
\begin{equation}\label{conv_est}
\lambda_1(p,t)\geq -\eps r^{-1}.
\end{equation}
\end{theorem}

\begin{proof}
The argument is by selecting a sequence of counterexamples that avoids the post-surgery time-slices, similarly as in \cite{HK}. Since this has do be done somewhat carefully, let us spell out the details:

Fix $\alpha>0$, and let $\bar{\delta}=\bar{\delta}(\alpha,M)>0$ be small enough. Let $\eps_0\in [0,\alpha^{-1}]$ be the infimum of $\eps$'s for which there are some constants $\bar{r}$ and $\eta$ such that the assertion of the theorem holds, and suppose towards a contradiction that $\eps_0>0$.

It follows that there is a sequence $\{\mathcal{K}^j \}$ of $(\alpha,\delta_j )$-flows, $\delta_j\leq\bar{\delta}$, in $P(p_j,t_j,jr_j)$ such that $H(p_j,t_j) \leq r_j^{-1}\to \infty$, but $\tfrac{\lambda_1}{H}(p_j,t_j) \to -\eps_0$ as $j\to\infty$. By the choice of $\eps_0$, it must be the case that $r_j H(p_j, t_j) \to 1$ as $j\to \infty$, since otherwise we could build a new sequence where $\tfrac{\lambda_1}{H}$ tends to something strictly smaller than $-\eps_0$.

By Theorem \ref{thm_loc_curv} we have bounds for $A$ and its space-time derivatives in $P(p_j,t_j,\tfrac{\rho}{2} r_j)$. Suppose there is no $\gamma > 0$ such that the flow is unmodified by surgeries in $P(p_j,t_j,\gamma r_j)$ after passing to a subsequence. We may assume (after wiggling a bit by factors tending to $1$ as $j\to \infty$), that $t_j$ is a surgery time, and that $(p_j,t_j)$ lies in $\partial K_{t_j}^\sharp\cap B(p_j, 5\Gamma s_j)$, c.f. \cite[Def. 2.4]{HK}. The radius of the surgery neck $s_j$ is comparable to $r_j$, again by \cite[Def. 2.4]{HK}, in particular we see that $s_j\to 0$. Thus, by item (3') of the definition of replacing a neck by caps, and since $\delta'(s_j)\to 0$ as $j\to \infty$, after passing to some point at controlled distance in the pre-surgery manifold we may assume that $(p_j,t_j)$ lies in the presurgery manifold $\partial K_{t_j}^-$, $H(p_j, t_j) = r_j^{-1}$, and $\tfrac{\lambda_1}{H}(p_j,t_j) \to -\varepsilon_0$ as $j\to \infty$.

After modifying the sequence as described, the argument can be concluded as in \cite[Proof of Thm.
1.8]{HK_meanconvex}. Namely, using Theorem \ref{thm_loc_curv} we can pass to a blow up limit $\mathcal{K}^\infty$ which is a smooth
mean curvature flow defined in a parabolic ball $P (0, 0, r)\subset\mathbb{R}^3\times\mathbb{R}$ such that the
ratio $\lambda_1/H$ attains a negative minimum $-\eps_0$ at $(0, 0)$; this contradicts the 
strict maximum principle.
\end{proof}

\begin{theorem}[{Global convergence theorem, c.f. \cite[Cor. 2.30]{HK}}]\label{thm_glob_conv}
There exists $\bar{\delta}=\bar{\delta}(\alpha,M)>0$ with the following property. If $\mathcal{K}^j$ is a sequence of $(\alpha,\delta_j)$-flows ($\delta_j\leq \bar{\delta}$) in $M$ and $(p_j,t_j)\in\partial \mathcal{K}^j$ is a sequence of points with $H(p_j,t_j)\to\infty$ then,\footnote{We tacitly assume the flows are defined at least for $t\in [t_j-\eps,t_j]$ for some $\eps>0$.} after passing to a subsequence and discarding connected components in $B(p,\Lambda_j H^{-1}(p_j,t_j))$ that do not contain $p_j$ for a suitable sequence $\Lambda_j\to \infty$, the sequence $\hat{K}_t^j=H(p_j,t_j)\cdot \exp_{p_j}(K_{t_j+H^{-2}(p_j,t_j)t})$
converges smoothly and globally (c.f. \cite[Def. 2.31]{HK}) to a limit $\mathcal{K}^\infty=\{K_t^\infty\subset\mathbb{R}^3\}_{t\in (-\infty,0]}$ which is a generalized $(\alpha,\delta)$-flow (see \cite[Def. 2.29]{HK}) in Euclidean three-space with convex time slices.
\end{theorem}

\begin{proof}
Using Theorem \ref{thm_loc_curv} and Theorem \ref{thm_conv} instead of \cite[Thm. 1.6, Thm. 1.8]{HK}, the proof from \cite{HK} goes through.
\end{proof}

In particular, note that blowups limits from Theorem \ref{thm_glob_conv} (as well as the standard solution) are always defined in Euclidean space and thus the discussion of their properties from \cite[Sec. 3]{HK} applies literally.

Finally, let us revisit the existence theorem and the canonical neighborhood theorem. The definition of $(\Balpha,\delta,\mathbb{H})$-flows is as in \cite[Def. 1.17]{BHH} with the following modifications:
\begin{itemize}
\item [i.] The parameter $\alpha$ is a parameter of the whole flow and not just of the initial domain.
\item [ii.] We assume in addition that $\inf H\geq 4\alpha/\textrm{inj}(M)$.
\item [iii.] We can simply set $\beta=1$.
\end{itemize}

\begin{theorem}[{Canonical neighborhood theorem, c.f. \cite[Thm. 1.22]{HK}}]\label{thm_canonical}
For every $\eps>0$ there exists $H_{\textrm{can}}<\infty$ such that if $\mathcal{K}$ is an $(\Balpha,\delta,\mathbb{H})$-flow with
$\delta$ small enough and $H_{\textrm{trig}}\gg H_{\textrm{neck}}\gg H_{\textrm{th}}\gg 1$, then every $(p,t)\in\partial\mathcal{K}$ with $H(p,t)\geq H_{\textrm{can}}$ is $\eps$-close\footnote{This is defined via the exponential map similarly as in Definition \ref{def_strong_neck}.} to either (a) an ancient $\alpha$-Andrews flow in $\mathbb{R}^3$ or (b) the evolution of a standard cap preceded by the evolution of a round cylinder $\bar{D}^2\times\mathbb{R}\subset \mathbb{R}^3$.
\end{theorem}

\begin{proof}
If not, then there is a sequence $\mathcal{K}^j$ of $(\Balpha,\delta_j,\mathbb{H}_j)$-flows ($\delta_j\leq \bar{\delta}$) with $H_{\textrm{trig}}^j/ H_{\textrm{neck}}^j,H_{\textrm{neck}}^j/ H_{\textrm{th}}^j\to \infty$ and a sequence of points $(p_j,t_j)\in\partial\mathcal{K}^j$ around which the flow (modulo rescaling) is not $\eps$-close to any model from (a) or (b). However, by Theorem \ref{thm_glob_conv} after passing to a subsequence the rescaled flows $\hat{\mathcal{K}}^j$ converge smoothly and globally to a limit $\mathcal{K}^\infty$ which is a a generalized $(\alpha,\bar{\delta})$-flow in $\mathbb{R}^3$. If $\mathcal{K}^\infty$ doesn't contain surgeries, then it is an ancient $\alpha$-Andrews flow in $\mathbb{R}^3$. If $\mathcal{K}^\infty$ contains surgeries,  then arguing similarly as in \cite[Sec. 4.1]{HK} we see that $\mathcal{K}^\infty$ must be the evolution of a standard cap preceded by the evolution of a round cylinder in $\mathbb{R}^3$. In either case, for $j$ large enough this gives a contradiction with not being $\eps$-close to one of the models.\footnote{Other potential connected components clear out similarly as in \cite[Cor. 2.15]{HK_meanconvex}.}
\end{proof}

\begin{theorem}[{Existence theorem, c.f. \cite[Thm. 2.21]{HK}}]\label{thm_existence}
Let $K\subset M^3$ be a mean convex domain. Then for every $T<\infty$, choosing $\delta$ small enough and $H_{\textrm{trig}}\gg H_{\textrm{neck}}\gg H_{\textrm{th}}\gg 1$, there exists an $(\Balpha,\delta,\mathbb{H})$-flow $\{K_t\}_{t\in [0,T]}$ with initial condition $K_0=K$.
\end{theorem}

\begin{proof}Let us start with some basic a priori estimates for $(\Balpha,\delta,\mathbb{H})$-flows in general ambient manifolds.
By compactness, the initial domain $K$ satisfies $H\geq H_0>0$ for some $H_0>0$ (and also $H\leq \gamma$ for some $\gamma<\infty$) and is $\alpha_0$-noncollapsed for some $\alpha_0>0$.

By the evolution equation for the mean curvature \cite[Cor. 3.5]{Hui86},
\begin{equation}
\partial_t H = \Delta H + |A|^2 H + \textrm{Rc}(\nu,\nu)H,
\end{equation}
along smooth mean curvature flow we get the lower bound
\begin{equation}
H\geq H_0e^{-\rho t},
\end{equation}
where $\rho$ is a bound for the Ricci curvature of the ambient manifold. Since the minimum of the mean curvature doesn't decrease under surgeries, we in fact get the a priori estimate
\begin{equation}
\inf H\geq H_0e^{-\rho T},
\end{equation}
for any $(\Balpha,\delta,\mathbb{H})$-flow $\{K_t\}_{t\in [0,T']}$ which is defined on an interval $[0,T']$ with $T'\leq T$ and has initial condition $K_0=K$.

Similarly, by an estimate of Brendle \cite{Brendle_noncoll_mfd} the noncollapsing factor deteriorates at most exponentially in time. Using this and the fact that the surgery caps have a controlled noncollapsing constant, we see that for small enough $\alpha>0$ every $(\Balpha,\delta,\mathbb{H})$-flow $\{K_t\}_{t\in [0,T']}$ which is defined on an interval $[0,T']$ with $T'\leq T$ and has initial condition $K_0=K$ is in fact an $(\Balpha',\delta,\mathbb{H})$-flow with $\alpha'=2\alpha$. After possibly decreasing $\alpha$ further, we can assume in addition that $8\alpha\leq {\textrm{inj}(M)}H_0e^{-\rho T}$.

As a final modification, let us explain how to replace necks by caps in general ambient manifolds. To do this, we simply map to Euclidean space via the exponential map, replace the neck in Euclidean space by standard caps as in \cite[Prop. 3.10]{HK} and map back to the manifold via the logarithm map. This obviously satisfies properties (1) and (2) of \cite[Def. 2.4]{HK}. Moreover, since the necks are at smaller and smaller scales as $\delta$ decreases, it is also clear that properties (3') and (4') are satisfied for some function $\delta'(x)$ that tends to zero as $x$ tends to zero.

Using the above, and Theorem \ref{thm_canonical} instead of \cite[Thm. 1.22]{HK}, we can now prove existence of $(\Balpha,\delta,\mathbb{H})$-flows in general ambient manifolds via the same continuity argument as in \cite[Sec. 4.2]{HK}.\footnote{Observe in particular that \cite[Claim 4.7]{HK} holds, since $H_{\textrm{neck}}^j\to \infty$ makes the ambient space look more and more Euclidean at the scales of the necks.}
\end{proof}

\section{Smooth foliations in general ambient manifolds}\label{sec_fol_amb}

The goal of this section is prove Theorem \ref{main_foliation_theorem}, which we restate here more precisely:

\begin{theorem}
Let $K\subset M^3$ be a smooth 3-disc with mean convex boundary. Then one of the following alternatives holds true:
\begin{enumerate}
\item There exists a stable embedded minimal $2$-sphere $\Sigma\subset \textrm{Int}(K)$.
\item There exists a smooth foliation $\{ \Sigma_t \}_{t\in [0,1]}$ of $K$ by mean convex embedded 2-spheres. More precisely:
\begin{itemize}
\item $\{\Sigma_t\}_{t\in [0,1)}$ is a smooth family of mean convex embedded 2-spheres.
\item $\cup_{t\in [0,1]}\Sigma_t = K$
\item $\Sigma_{t_1}\cap \Sigma_{t_2}=\emptyset$ whenever $t_1\neq t_2$.
\item $|\Sigma_{t_2}|<|\Sigma_{t_1}|$ whenever $t_2>t_1$.
\item $\Sigma_0=\partial K$
\item For $t\to 1$ the spheres $\Sigma_t$ converge in the Hausdorff sense to a limit $\Sigma_1$, and $\Sigma_1$ is a finite union of smooth arcs.
\end{itemize}
\end{enumerate}
\end{theorem}

\begin{remark}
With some additional effort it is possible to produce a smooth foliation by mean convex embedded 2-spheres which has the additional property that it terminates in a round point $p$ (and thus in particular $\Sigma_1=\{p\}$), but this is not needed for our application.
\end{remark}

\begin{proof}
Let $K\subset M^3$ be a smooth 3-disc with mean convex boundary.

Consider the level set flow $\{K_t\}_{t\geq 0}$ starting at $K_0=K$. Let 
\begin{equation}
K_\infty:=\bigcap_{t\geq 0} K_t.
\end{equation}

Consider first the case $K_\infty\neq\emptyset$. By a result of White \cite{White_size} the boundary of $K_\infty$ is a union of finitely many stable minimal surfaces. Since the genus is monotone under level set flow \cite{White_topology},\footnote{Since $K_t$ is smooth for a.e. time the genus is well defined for $a.e.$ time.} these minimal surfaces must be spheres, i.e. in the case $K_\infty\neq\emptyset$ we can find a stable minimal 2-sphere in the interior of $K$.

Assume from now on that $K_\infty=\emptyset$, i.e. that the level set flow becomes extinct in finite time. Observe first that any mean curvature flow with surgery starting at $K$ gives a family of closed sets that is a set theoretic subsolution in the terminology of Ilmanen \cite{Ilmanen}. Since the level set flow is the maximal set theoretic subsolution, we get an a priori bound $T<\infty$ for the extinction time of any mean curvature flow with surgery starting at $K$.

We can now choose the surgery parameters suitably so that in particular both the existence theorem (Theorem \ref{thm_existence}) and the canonical neighborhood theorem (Theorem \ref{thm_canonical}) apply. By choosing the curvature parameters large enough we can ensure that at the neck scale $H_{\textrm{neck}}^{-1}$ the ambient space looks as close as we want to Euclidean space. The gluing map $\mathcal{G}_{r_s}$ can then simply be defined by mapping to Euclidean space with the exponential map, using the gluing map from \cite[Thm. 4.1]{BHH} in Euclidean space, and mapping back with the logarithm. Since the ambient space looks as close as we want to Euclidean space at the neck scale the construction of the foliation from Section \ref{sec_fol_eucl} applies.

The only little point that hasn't been explained yet is that $|\Sigma_{t_2}|<|\Sigma_{t_1}|$ whenever $t_2>t_1$, but this follows easily from the first variation formula given that the foliation is mean convex.
\end{proof}

\section{Ellipsoids and Yau's question}\label{sec_ell}

In this section, we illustrate our results in detail in the special case of ellipsoids, and answer Yau's question.  We then discuss the relationship with the multiplicity one conjecture and the equidistribution of width conjecture.

Since it is difficult to tell whether the ellipsoids are bumpy, we need to generalize Theorem \ref{main} to the case of non-generic metrics:
\begin{theorem}\label{nongeneric}
Let $M$ be a three-manifold diffeomorphic to a three-sphere not containing any stable minimal two-spheres.  Suppose $M$ admits an optimal foliation $\{\Lambda_t\}_{t\in [-1,1]}$ so that $\Sigma_1:=\Lambda_0$ is an index $1$ minimal two-sphere realizing the $1$-width $\omega_1(M)$.  Then either $M$ contains infinitely many minimal two-spheres of area $|\Sigma_1|$, or else $M$ contains a minimal two-sphere $\Sigma_2$ with area $\omega_2(M)=|\Sigma_2|$ satisfying:
\begin{equation}
|\Sigma_1|<|\Sigma_2|<2|\Sigma_1|.
\end{equation}
\end{theorem}
The proof of Theorem \ref{nongeneric} follows immediately from considering the proof of Theorem \ref{main}.

Let us know turn to the ellipsoids.  For two-dimensional ellipsoids with different major axis lengths, the three geodesics provided by Lusternik-Schnirelmann coincide with the intersections of the ellipsoid with the coordinate planes. Moreover, these three are the only three closed embedded geodesics. One dimension higher, there is a remarkably different phenomenon when the ellipsoid becomes very elongated, which we will now explain.

Recall from the introduction, given $a> b> c> d>0$, we consider the ellipsoid
\begin{equation}
E(a,b,c,d):= \left\{\frac{x_1^2}{a^2}+ \frac{x_2^2}{b^2}+\frac{x_3^2}{c^2}+\frac{x_4^2}{d^2}=1\right\}\subset\mathbb{R}^4.
\end{equation}
For each $i=1,2,3,4$ denote by
\begin{equation}
\Gamma_i:= E(a,b,c,d)\cap \{x_i=0\}.
\end{equation}
the minimal `planar' two-sphere.  Note that the areas of $\Gamma_i$ are increasing in $i$. From now on we fix $b,c,d$, and consider $E(a):=E(a,b,c,d)$.\\
We will now prove Theorem \ref{yau}, which we restate:
\begin{theorem}[Answer to Yau's Question (1987)]
For $a\gg b$, $E(a)$ contains an embedded non-planar minimal two-sphere.
\end{theorem}
\begin{proof}
Let $\Sigma_1(a)$ denote the $1$-width of $E(a)$.   We can assume $\Sigma_1(a)=\Gamma_1(a)$ (otherwise the theorem is proved).  
Thus from Lemma \ref{lemma_width} we know $\Sigma_1(a)$ is a smallest area embedded minimal two-sphere, has index $1$, and sits in an optimal foliation $\{\Lambda_t\}_{t\in [-1,1]}$ with $\Sigma_1(a)=\Lambda_0$ (that is easy to construct explicitly).   

By Theorem \ref{nongeneric} either $E(a)$ contains infinitely many two-spheres of area $\Sigma_1(a)$, or else a minimal two-sphere $\Sigma_2(a)$ satisfying
\begin{equation}\label{arearestriction}
|\Gamma_1(a)|<|\Sigma_2(a)|<2|\Gamma_1(a)|.
\end{equation}

In the first case where $E(a)$ contains infinitely minimal two-spheres of area $|\Sigma_1(a)|$, the theorem is proved as these cannot all be planar since when $a<b<c<d$, $E(a)$ contains only four planar minimal two-spheres. 

In the second case we obtain a minimal two-sphere $\Sigma_2(a)$ satisfying \eqref{arearestriction}.  If $\Sigma_2(a)$ were planar, then \eqref{arearestriction} would imply
\begin{equation}\label{areafinal}
|\Gamma_2(a)|<2|\Gamma_1(a)|.
\end{equation}
However, when $a\gg b$ (i.e. for very elongated ellipses) we can see directly that \eqref{areafinal} fails.  

The area of $\Gamma_1(a)$ is given by some function $F(b,c,d)$ obtained by elliptic integrals.  For $a$ much larger than the other parameters, $\Gamma_2(a)$ resembles a cylinder of length $a$ and cross section the ellipse in $\mathbb{R}^2$ given by $x_3^2/c^2+x_4^2/d^2=1$.  The length of this cross section is again a function $G(c,d)$ obtained by elliptic integrals. Thus, the area of $\Gamma_2(a)$ is approximately $a\cdot G(c,d)$.  Letting $a\rightarrow\infty$, and holding the other parameters fixed, we see that \eqref{areafinal} fails for $a$ large enough.

This completes the proof.
\end{proof}

Another reason to expect that when $a$ is large, the second width of $M$ is not realized by $\Gamma_2(a)$ comes from estimating the Morse index.  Combined with the index estimates of Proposition \ref{ellipsoidsprop}, the following proposition gives a second proof of Theorem \ref{yau}.

\begin{proposition}\label{index}
If  $a\gg b$, then the index of the second planar sphere $\Gamma_2(a)\subset E(a)$ is greater than $2$. 
\end{proposition}

\begin{proof}
We fix $b$, $c$ and $d$ and consider $\Gamma_2(a)=\Gamma_2(a,b,c,d)$ as a function of $a$.
Explicitly, we have:
\begin{equation}\label{e}
\Gamma_2(a)=  \left\{\frac{x_1^2}{a^2}+ \frac{x_3^2}{c^2}+\frac{x_4^2}{d^2}=1\right\}\subset\mathbb{R}^4.
\end{equation}
Note that $\Gamma_2(a)$ converges as $a\rightarrow \infty$ to $\mathbb{R}\times F$, where
\begin{equation}
F=\left\{\frac{x_3^2}{c^2}+\frac{x_4^2}{d^2}=1\right\}\subset\mathbb{R}^2.
\end{equation}
Recall that the index of $\Gamma_2(a)$ is the maximal dimension of a subspace on which $L:=L_{\Gamma_2(a)}=-\Delta-|A|^2-\textrm{Ric}(\nu,\nu)$ is negative definite.  To show that the index is larger than 2 for $a$ large enough, we will find three functions $\phi_1$, $\phi_2$ and $\phi_3$ on $\Gamma_2(a)$ with disjoint support so that for $i=1,2,3$ we have
\begin{equation}\label{neg}
\int_{\Gamma_2(a)} \phi_i L\phi_i < 0.
\end{equation}
To this end, from \eqref{e} and the convergence $\Gamma_2(a)\rightarrow \mathbb{R}\times F$ we infer that there exists $\mu>0$ such that 
\begin{equation}\label{eq_riccilower}
|A|^2+\textrm{Ric}(\nu,\nu)>\mu \mbox{ as } a\rightarrow\infty \mbox{ on any compact subset of } \mathbb{R}^4.  
\end{equation}
\noindent
Fix $N> \max(\frac{1}{\mu},2)$. For $A,B\in \mathbb{R}$ with $A<B$ let us define the following linear cutoff function on $\mathbb{R}^4$: 

\begin{equation*}
    \phi_{A,B}(x_1,x_2,x_3,x_4) = \begin{cases}
 		0              & x_1\leq A-1 \\   
    
                x_1+1-A               & A-1\leq x_1\leq A \\
                	1				& A\leq x_1\leq B	\\
             -x_1+B+1                 & B\leq x_1\leq B+1\\
           0     & x_1>B+1.\
           \end{cases}
\end{equation*}

Let $\phi_{A,B}$ also refer to the restriction of $\phi_{A,B}$ to $\Gamma_2(a)$.   We will show that  taking $a$ large enough, the functions
\begin{equation}
\phi_1:=\phi_{-4N,-2N},\quad \phi_2:=\phi_{-N,N},\quad \phi_3:= \phi_{2N,4N}
\end{equation}
defined on $\Gamma_2(a)$ satisfy \eqref{neg}. Note also that the functions $\phi_1,\phi_2,\phi_3$ have disjoint support. Let us compute for $\phi_2$ (the analogous computation works for $\phi_1$ and $\phi_3$):
\begin{align}\label{comp_ind}
\int_{\Gamma_2(a)} \phi_2 L\phi_2 &= \int_{\Gamma_2(a)} |\nabla\phi_2|^2- \int_{\Gamma_2(a)}(|A|^2+\textrm{Ric}(\nu,\nu))\phi_2^2\nonumber\\
&\leq |\Gamma_2(a)\cap \{N\leq |x_1|\leq N+1\}|\nonumber\\
&\qquad\qquad-\int_{\Gamma_2(a)\cap \{|x_1|\leq N\} }(|A|^2+\textrm{Ric}(\nu,\nu)).
\end{align}
As $a\rightarrow\infty$ we have
\begin{equation} \label{1}
 |\Gamma_2(a)\cap \{N\leq |x_1|\leq N+1\}|\rightarrow 2|F|,
\end{equation}
and
\begin{equation} \label{3}
|\Gamma_2(a)\cap \{|x_1|\leq N\} |\rightarrow 2N|F|.
\end{equation}
Using \eqref{eq_riccilower}, \eqref{comp_ind}, \eqref{1}, \eqref{3} and the choice of $N$ we conclude that
\begin{equation}
\limsup_{a\rightarrow\infty}\int_{\Gamma_2(a)} \phi_2 L\phi_2 \leq 2|F|-2N\mu |F| <0.
\end{equation}

This completes the proof.
\end{proof}

We obtain the following corollary:
\begin{corollary}\label{moduli}
For some critical value of $a$, $E(a)$ contains an embedded minimal two-sphere with a non-trivial Jacobi field that does not arise from deformation of the minimal surface by ambient isometry. 
\end{corollary}

It is a long-standing question of Yau whether a three-manifold with positive Ricci curvature can contain a continuously varying family of minimal surfaces.  One way to hope to prove that it cannot occur is to rule out non-trivial Jacobi fields.  Corollary \ref{moduli} shows that such fields may in fact exist.  
\begin{remark}
In fact, the proof of Proposition \ref{index} shows that the Morse index of $\Gamma_2(a)$ approaches infinity as $a\rightarrow\infty$.
\end{remark}

\begin{remark}
In the case of Lusternik-Schnirelmann geodesics on two-dimensional ellipsoids, the second planar geodesic converges to two stable lines on a flat cylinder as $a\rightarrow \infty$.   The above argument to construct many negative variations fails, since the curvature term appearing in the second variation formula for length is the Gaussian curvature which vanishes on the two-dimensional cylinder.   Needless to say, the catenoid estimate argument to show $\omega_2< 2\omega_1$ also fails in two dimensions as the ``neck" one would add would contribute significant length.
\end{remark}

\subsection{Multiplicity One Conjecture}
Marques-Neves have made the following multiplicity one conjecture \cite{MN_mult}:  
\begin{conj}[Marques-Neves]
Given a min-max procedure on a compact manifold, if the metric is bumpy then any two-sided unstable component of the min-max limit must occur with multiplicity one. 
\end{conj}
The ellipsoids discussed in the previous section do not violate this conjecture, but show that it is sharp in the sense that as metrics degenerate to a non-compact limit one can have multiplicity in the limit.  Namely, assuming recently announced work of Marques-Neves, we obtain the following proposition.  Recall that $E(a)\rightarrow P\times\mathbb{R}$, where $P$ is a two dimensional ellipsoid:

\begin{proposition}[Ellipsoids]\label{ellipsoidsprop}
For $a$ large enough, any minimal two-sphere $\Sigma_1(a)$ realizing the $1$-width of $E(a)$ has index $1$ and nullity $0$.  The two-width $\omega_2$ of $E(a)$ is realized by a minimal sphere $\Sigma_2(a)$ with index $2$ or else with index $1$ and non-trivial nullity. In either case, 
\begin{equation}
 |\Sigma_2(a)|<2|\Sigma_1(a)|.  
 \end{equation}
 Moreover, 
\begin{equation}
\Sigma_1(a)\rightarrow P\times\{0\}\mbox{ in the sense of varifolds as } a\rightarrow\infty
\end{equation}
and
\begin{equation}
\Sigma_2(a)\rightarrow 2(P \times\{0\})\mbox{ in the sense of varifolds as } a\rightarrow\infty
\end{equation}
\end{proposition}


\begin{remark}
Loosely speaking, for $a$ large, $\Sigma_2(a)$ resembles two parallel copies of $\Gamma_1(a)$ joined by a small neck. An interesting problem is to study where precisely the neck is located, and whether it is rigid.  \end{remark}
\begin{proof}

 Note that the only finite area minimal surfaces in $P\times\mathbb{R}$ are of the form $P\times\{t\}$ for some $t\in\mathbb{R}$ (this follows from the maximum principle and the monotonicity formula). 


Let $V$ denote a varifold limit of $\Sigma_1(a)$ as $a\rightarrow\infty$.  Since $E(a)$ has positive Ricci curvature, it follows by Frankel's theorem that $\Sigma_1(a)$ contains a point in $P\times\{0\}$.  Since each $\Sigma_1(a)$ contains a point in $P\times\{0\}$, by the monotonicity formula, $V$ contains a point in $P\times\{0\}$, and in particular $V$ is not the empty set.  By the area bounds, 
\begin{equation}\label{firstwidth}
|\Sigma_1(a)|\leq |\Gamma_1(a)|, 
\end{equation}
it follows from Choi-Schoen \cite{ChoiSchoen} that $V$ is a smooth minimal surface (potentially with multiplicity).  Again by the area bound \eqref{firstwidth}, it follows that the multiplicity with which $V$ occurs is $1$ and thus $V= P\times\{0\}$ as claimed.  The index of $\Sigma_1(a)$ is $1$ by Lemma \ref{lemma_width}.  To see that $\Sigma_1(a)$ does not have nullity, observe that by the smooth convergence $\Sigma_1(a)\rightarrow V$, if it had nullity, and since $V$ is stable, the nullity of $V$ would be at least $2$.  This is a contradiction as the nullity of $V$ in $P\times\mathbb{R}$ is $1$.

Let us now consider the index and nullity of the second minimal two-sphere $\Sigma_2(a)$.  Let $g_i$ be a sequence of bumpy metrics approaching $E(a)$.  Note that we can assume that $g_i$ have positive Ricci curvature and thus do not admit any stable two-spheres.  Thus Theorem \ref{main} implies that the $2$-width of the metric $g_i$ is realized by an index $1$ or $2$ minimal surface $\Sigma_2^i(a)$ and is obtained with multiplicity $1$.  Marques-Neves have announced that in this case, one can assume that the index of $\Sigma_2^i(a)$ is equal to the number of parameters used in the sweepouts to produce it, and thus is $2$.  Thus as $i\rightarrow\infty$ the sequence $\Sigma_2^i(a)$ converges to a minimal sphere $\Sigma_2(a)$ in $E(a)$ with index $2$ or else index $1$ and non-trivial nullity (a stable surface is ruled out as $E(a)$ admits no such surfaces).  Moreover, $\Sigma_2(a)$ realizes the $2$-width of $E(a)$.  Note that the convergence is with multiplicity $1$ (otherwise, one obtains a positive Jacobi field in violation of the fact that $E(a)$ contains no stable two-spheres).

Now let $V$ denote a varifold limit of $\Sigma_2(a)$ as $a\rightarrow\infty$. Since each $\Sigma_2(a)$ contains a point in $P\times\{0\}$, again by the monotonicity formula, $V$ contains a point in $P\times\{0\}$, and in particular $V$ is not the empty set.  

We claim $V$ is the varifold $P\times\{0\}$ counted with multiplicity $2$. Since the surfaces $\Sigma_2(a)$ have bounded area and genus, and stay in a compact region in $P\times\mathbb{R}$, it follows by Choi-Schoen \cite{ChoiSchoen} that $V$ is a smooth minimal surface of finite area, potentially counted with multiplicity $k$.  Since the only finite area smooth minimal surfaces are of the form $P\times \{t\}$ and $V$ contains a point in $P\times \{0\}$, it follows that $V$ is the surface $P\times \{0\}$ counted with some integer multiplicity.  By the width bound and the earlier part of this proposition,
\begin{equation}
|\Sigma_2(a)|<2|\Sigma_1(a)|\rightarrow 2|P\times\{0\}|
\end{equation}
we see that $k$ is either $1$ or $2$.  Finally observe that $k$ cannot be equal to $1$, since in that case the local regularity theorem would imply smooth convergence.  By converge of eigenvalues of the stability operator under smooth convergence, the two negative eigenvalues of $\Sigma_2(a)$ (or one negative and one zero eigenvalue) converge to zero as $a\rightarrow\infty$.  But the nullity of $P\times\{0\}$ is $1$, not $2$.  This is a contradiction.  Thus $\Sigma_2(a)\rightarrow 2(P \times\{0\})$ as claimed.
\end{proof}

\begin{remark}
One way to understand this example is to observe that the space of essential embedded two-spheres in $\mathbb{S}^2\times\mathbb{R}$ retracts onto a single two-sphere.  Since $E(a)$ (which as a topological $\mathbb{S}^3$ has interesting four parameter families of embedded two-spheres) is converging to a space with much less topology in the space of embedded two-spheres it is reasonable that one has fewer options for the limiting two-sphere and thus multiplicity is forced. Morally, this is somewhat reminiscient of the fact that multiplicity can occur for mean curvature flow as $t\to \infty$.
\end{remark}

\subsection{Equidistribution of Widths}
In \cite{MN} Marques-Neves consider applying min-max theory to the space of $2$-cycles with $\mathbb{Z}_2$ coefficients in a three-manfiold.  This space is homotopy equivalent to $\mathbb{RP}^\infty$ and thus one can consider non-trivial $k$ parameter sweepouts for any positive integer $k$ (corresponding to non-trivial $\mathbb{RP}^k\subset \mathbb{RP}^\infty$).  They considered the corresponding $k$-widths $\omega_k$ and conjectured that the minimal surfaces $\Sigma_k$ realizing the $k$-widths should become equidistributed in the ambient manifold as $k\rightarrow\infty$.  

The elongated ellipsoids provide an interesting family of examples which do not disprove the conjecture but show that given any positive integer $k$, one can find a manifold all of whose first $k$-widths are realized by minimal surfaces arbitrarily close in the Hausdorff topology to a single minimal surface (and thus are not equistributed). Namely, we have the following:
\begin{proposition}
Given any $k\in\mathbb{N}$, and $\eps>0$, there exists $a(k, \eps)>0$ so that whenever $a>a(k, \eps)$ the minimal surfaces $\{\Sigma_j(a)\}_{j=1}^k$ realizing the first $k$-widths in $E(a)$ are contained in an $\eps$-neighborhood (in Hausdorff topology) of $P\times\{0\}$.
\end{proposition}
\begin{proof}
Suppose $k$ is given and fixed.  Let $\omega_k(a)$ denote the $k$-width in $E(a)$ realized by a minimal surface $\Sigma_k(a)$ (potentially with multiplicity).
Note that we have the trivial linear bound
\begin{equation}\label{d}
\omega_k(a)\leq k|P\times\{0\}|
\end{equation}
for all $a$. If the assertion of the proposition failed, there would be a sequence of $a_i\rightarrow\infty$ and some integer $j$ satisfying $1\leq j\leq k$ so that $\Sigma_j(a_i)$ is not contained in any $\eps$-neighborhood of $P\times\{0\}$.  By the area bounds \eqref{d}, it follows that we can take a subsequential limit $V$ of $\Sigma_j(a_i)$ that is a stationary integral varifold with bounded area. But as in the previous section, $\Sigma_j(a_i)$ intersect $P\times\{0\}$ and so $V$ cannot vanish by the monotonicity formula and moreover contains a point in $P\times\{0\}$.  But the only stationary integral varifold with bounded area in $P\times\mathbb{R}$ is a slice of the form $P\times\{t\}$ for some $t\in\mathbb{R}$, and we necessarily have $t=0$ in our case.  This is a contradiction.
\end{proof}

\section{Minimal embedded projective planes in $\mathbb{RP}^3$}\label{section_projective}
In this final section, we prove the following theorem:  
\begin{theorem}[Projective Planes in $\mathbb{RP}^3$]\label{rp2s}
Let $M$ be a $3$-manifold diffeomorphic to $\mathbb{RP}^3$ endowed with a metric of positive Ricci curvature.  Then $M$ admits at least two embedded minimal projective planes.  
\end{theorem}
Unlike in the study of two-spheres in the three-sphere, we do not need the mean curvature flow in Theorem \ref{rp2s} and the argument only uses min-max theory.  On the other hand, we do need to assume positive Ricci curvature as otherwise we could obtain the existence of a minimal projective plane and potentially a minimal two-sphere.  The result can also be interpreted as a kind of non-orientable analog to Theorem 3.3 in \cite{KMN}.
\begin{proof}
Let $\mathcal{P}$ denote the space of projective planes embedded in $M$ and denote
\begin{equation}\label{minimum}
\mathcal{A}(M):=\inf_{\Sigma\in\mathcal{P}} |\Sigma|.
\end{equation} 
The infimum  in \eqref{minimum} is achieved by an embedded minimal projective plane $\Sigma_0$ (see for instance Proposition 5 in \cite{BBEN}).  We will now use $\Sigma_0$ and a one-parameter min-max argument to produce a second minimal projective plane.

Since $M$ has positive Ricci curvature, there is a double cover $\tilde{M}$ of $M$ in which the lift $\tilde{\Sigma}_0$ of $\Sigma_0$  is an unstable minimal two-sphere bounding $3$-balls $B_\pm$ on both sides.  Note that $B_+$ and $B_-$ are interchanged by the deck group of the covering, and that $\tilde{\Sigma}_0$ is fixed by the group.

Since $\tilde{M}$ has positive Ricci curvature, we can easily obtain an optimal foliation $\{\tilde{\Sigma}_t\}_{t\in [0,1]}$ of $B_+$ by embedded two-spheres. 
Now for $t\in [0,1]$ let
\begin{equation}
\tilde{\Gamma}_t:=\tilde{\Sigma}_t\cup\tilde{\Sigma}_0,
\end{equation}
and consider the set $\Gamma'_t$ obtained by projecting $\tilde{\Gamma}_t$ down to $M$.  Note that
\begin{equation}
\sup_{t\in [0,1]}|\Gamma'_t|=3|\Sigma_0|.
\end{equation}
Finally, for each $t$, we connect the two disjoint components of $\Gamma'_t$ by a very skinny tube whose area approaches $0$ as $t\rightarrow 0$ and thus we obtain a family of surfaces $\Gamma_t$ so that $\Gamma_t\rightarrow\Sigma_0$ as $t\rightarrow 0$.   For $t$ near $1$, using the catenoid estimate again we can open up the neck of the tube to retract $\Gamma_t$ to $\Sigma_0$.  It is not hard to see that this new family $\Gamma_t$ is a non-trivial family of projective planes.  After applying the catenoid estimate, we obtain that the amended family satisfies
\begin{equation}\label{catenoidmax}
\sup_{t\in [0,1]}|\Gamma_t|<3|\Sigma_0|.
\end{equation}

Considering the saturation of the family $\Gamma_t$ we can consider the associated min-max value and minimal surface $\Sigma_1$.  Since $M$ has positive Ricci curvature, $\Sigma_1$ is connected.   By the genus bounds obtained in \cite{K}, it follows that $\Sigma_1$ is itself a projective plane obtained with odd multiplicity.   If $\Sigma_1$ is not equal to $k\Sigma_0$ for some integer $k$ then we are done.  

Suppose that $\Sigma_1=k\Sigma_0$ for some odd integer.  By \eqref{catenoidmax}, $k<3$.   Thus the only possibility is that $k=1$ and $\Sigma_1=\Sigma_0$.  

However, in this case the following elementary Lusternik-Schnirelmann argument produces a second projective plane: Suppose there exists a sequence of pulled tight sweep-outs $\Gamma^i_t$ satisfying:
\begin{equation}\label{areasgoingdown}
\sup_{t\in\mathbb{S}^1}|\Gamma^i_t|\leq |\Sigma_0|+\eps_i
\end{equation}
for some sequence $\eps_i\rightarrow 0$.  

Let $\eps>0$.  For each $i$, we can consider the following subsets of $\mathbb{S}^1$
\begin{equation}
A_i:= \{t\in \mathbb{S}^1 \;|\; \mathcal{F}(\Gamma^i_t,\Sigma_0)<\eps\}
\end{equation}
and
\begin{equation}
B_i:= \{t\in\mathbb{S}^1 \;|\; \mathcal{F}(\Gamma^i_t,\Sigma_0)>\eps/2\}.
\end{equation}
Since the $\Gamma^i_t$ are non-trivial sweep-outs, it follows that $B_i$ is non-empty as long as $\eps$ is chosen small enough and $i$ is large enough.

Consider a min-max sequence $\Gamma^i_{t_i}$ obtained from $B_i$ (by \eqref{areasgoingdown} any sequence of parameters $t_i$ in $B_i$ will suffice).  If the min-max sequence $\Gamma^i_{t_i}$ is almost minimizing, then it has a smooth limit distinct from $\Sigma_0$ as it is obtained from the set $B_i$ a definite distance from $\Sigma_0$.   Thus we obtain a second minimal embedded projective plane $\Sigma_2$ with area equal to that of $\Sigma_0$.  

If instead $\Gamma^i_{t_i}$ is not almost minimizing, then for $i$ large we can decrease area of $\Gamma^i_{t_i}$ below $|\Sigma_0|$.  This contradicts the fact that $\Sigma_0$ realizes the infimum for the areas of projective planes in $M$.
\end{proof}

\bibliography{HaslhoferKetover}

\begin{thebibliography}{10}

\bibitem{ABN}
S.~Alexakis, T.~Balehowsky, and A.~Nachman.
\newblock private communication.

\bibitem{Almgren}
F.~Almgren.
\newblock Some interior regularity theorems for minimal surfaces and an
  extension of {Bernstein}'s theorem.
\newblock {\em Annals of Math.}, 84:272--292, 1966.

\bibitem{Birkhoff}
G.~Birkhoff.
\newblock Dynamical systems with two degrees of freedom.
\newblock {\em Trans. Amer. Math. Soc.}, 18(2):199--300, 1917.

\bibitem{BBEN}
H.~Bray, S.~Brendle, M.~Eichmair, and A.~Neves.
\newblock Area-minimizing projective planes in 3-manifolds.
\newblock {\em Comm. Pure and Appl. Math.}, 63(9):1237--1247, 2010.

\bibitem{Brendle_noncoll_mfd}
S.~Brendle.
\newblock An inscribed radius estimate for mean curvature flow in {R}iemannian
  manifolds.
\newblock {\em Ann. Sc. Norm. Super. Pisa Cl. Sci.}, 16(4):1447--1472, 2016.

\bibitem{BrendleHuisken}
S.~Brendle and G.~Huisken.
\newblock Mean curvature flow with surgery of mean convex surfaces in
  three-manifolds.
\newblock {\em J. Eur. Math. Soc. (to appear)}.

\bibitem{BHH}
R.~Buzano, R.~Haslhofer, and O.~Hershkovits.
\newblock The moduli space of two-convex embedded spheres.
\newblock {\em arXiv:1607.05604}, 2016.

\bibitem{CGG}
Y.~Chen, Y.~Giga, and S.~Goto.
\newblock Uniqueness and existence of viscosity solutions of generalized mean
  curvature flow equations.
\newblock {\em J. Differential Geom.}, 33(3):749--786, 1991.

\bibitem{ChoiSchoen}
H.~Choi and R.~Schoen.
\newblock The space of minimal embeddings of a surface into a three-dimensional
  manifold of positive {R}icci curvature.
\newblock {\em Invent. Math.}, 81(3):387--394, 1985.

\bibitem{CD}
T.~Colding and C.~De~Lellis.
\newblock The min-max construction of minimal surfaces.
\newblock {\em Surveys in Differential Geometry VIII}, pages 75--107, 2003.

\bibitem{CGK}
T.~Colding, D.~Gabai, and D.~Ketover.
\newblock On the classifications of heegaard splittings.
\newblock {\em arXiv:1509.05945}, 2015.

\bibitem{C}
O.~Cornea, G.~Lupton, J.~Oprea, and D.~Tanre.
\newblock {Lusternick}-{Schnirelman} category.
\newblock {\em Mathematical Surveys and Monographs, AMS}, 103, 2003.

\bibitem{ES}
L.~Evans and J.~Spruck.
\newblock Motion of level sets by mean curvature. {I}.
\newblock {\em J. Differential Geom.}, 33(3):635--681, 1991.

\bibitem{Grayson}
M.~Grayson.
\newblock Shortening embedded curves.
\newblock {\em Ann. of Math. (2)}, 129(1):71--111, 1989.

\bibitem{G}
L.~Guth.
\newblock The endpoint case of the {Bennett--Carbery--Tao} multilinear {Kakeya}
  conjecture.
\newblock {\em Acta Math.}, 205(2):263--286, 2010.

\bibitem{Ham_mon_amb}
R.~Hamilton.
\newblock Monotonicity formulas for parabolic flows on manifolds.
\newblock {\em Comm. Anal. Geom.}, 1(1):127--137, 1993.

\bibitem{HK_meanconvex}
R.~Haslhofer and B.~Kleiner.
\newblock Mean curvature flow of mean convex hypersurfaces.
\newblock {\em Comm. Pure Appl. Math.}, 70(3):511--546, 2017.

\bibitem{HK}
R.~Haslhofer and B.~Kleiner.
\newblock Mean curvature flow with surgery.
\newblock {\em Duke Math. J.}, 166(9):1591--1626, 2017.

\bibitem{Hatcher}
A.~Hatcher.
\newblock A proof of the {S}male conjecture, {${\rm Diff}(S^{3})\simeq {\rm
  O}(4)$}.
\newblock {\em Ann. of Math. (2)}, 117(3):553--607, 1983.

\bibitem{Hui84}
G.~Huisken.
\newblock Flow by mean curvature of convex surfaces into spheres.
\newblock {\em J. Differential Geom.}, 20(1):237--266, 1984.

\bibitem{Hui86}
G.~Huisken.
\newblock Contracting convex hypersurfaces in {R}iemannian manifolds by their
  mean curvature.
\newblock {\em Invent. Math.}, 84(3):463--480, 1986.

\bibitem{Ilmanen}
T.~Ilmanen.
\newblock Elliptic regularization and partial regularity for motion by mean
  curvature.
\newblock {\em Mem. Amer. Math. Soc.}, 108(520):x+90, 1994.

\bibitem{K}
D.~Ketover.
\newblock Genus bounds for min-max minimal surfaces.
\newblock {\em arXiv:1312.2666}, 2013.

\bibitem{KL}
D.~Ketover and Y.~Liokumovich.
\newblock in preparation.

\bibitem{KMN}
D.~Ketover, F.~Marques, and A.~Neves.
\newblock The catenoid estimate and its geometric applications.
\newblock {\em arXiv:1601.04514}, 2016.

\bibitem{LS}
L.~Lusternik and L.~Schnirelmann.
\newblock Topological methods in variational problems and their application to
  the differential geometry of surfaces.
\newblock {\em Uspehi Matem. Nauk (N.S.)}, 2(1(17)):166--217, 1947.

\bibitem{MN}
F.~Marques and A.~Neves.
\newblock Existence of infinitely many minimal hypersurfaces in positive
  {R}icci curvature.
\newblock {\em Invent. Math. (to appear)}.

\bibitem{MN2}
F.~Marques and A.~Neves.
\newblock Rigidity of min-max minimal spheres in three-manifolds.
\newblock {\em Duke Math. J.}, 161(14):2725--2752, 2012.

\bibitem{MN_mult}
F.~Marques and A.~Neves.
\newblock Morse index and multiplicity of min-max minimal hypersurfaces.
\newblock {\em Camb. J. of Math.}, 4(4):463--511, 2016.

\bibitem{SS}
F.~Smith.
\newblock On the existence of embedded minimal 2-spheres in the 3-sphere,
  endowed with an arbitrary {R}iemannian metric.
\newblock {\em Phd thesis, Supervisor: Leon Simon, University of Melbourne},
  1982.

\bibitem{White_posricci}
B.~White.
\newblock The space of minimal submanifolds for varying {R}iemannian metrics.
\newblock {\em Indiana Univ. Math. J.}, 40(1):161--200, 1991.

\bibitem{White_topology}
B.~White.
\newblock The topology of hypersurfaces moving by mean curvature.
\newblock {\em Comm. Anal. Geom.}, 3(1-2):317--333, 1995.

\bibitem{White_size}
B.~White.
\newblock The size of the singular set in mean curvature flow of mean convex
  sets.
\newblock {\em J. Amer. Math. Soc.}, 13(3):665--695, 2000.

\bibitem{White_bumpy}
B.~White.
\newblock On the bumpy metrics theorem for minimal submanifolds.
\newblock {\em arXiv:1503.01803}, 2015.

\bibitem{Yau}
{\relax S.T}.~Yau.
\newblock Nonlinear analysis in geometry.
\newblock {\em L'enseignement mathematique}, 34:109--158, 1987.

\end{thebibliography}

\bibliographystyle{abbrv}

\end{document}